\documentclass
{amsart}

\allowdisplaybreaks

\usepackage{color}
\usepackage{amsmath}
\usepackage{amssymb}
\usepackage[abbrev]{amsrefs}
\usepackage{stmaryrd}
\usepackage{graphicx}
\usepackage{bm}

\newtheorem{theorem}{Theorem}[section]
\newtheorem{proposition}[theorem]{Proposition}
\newtheorem{lemma}[theorem]{Lemma}

\theoremstyle{definition}
\newtheorem{definition}[theorem]{Definition}
\newtheorem{example}[theorem]{Example}

\theoremstyle{definition}
\newtheorem{remark}[theorem]{Remark}

\numberwithin{equation}{section}

\newcommand{\bfone}{{\mathbf{1}}}
\newcommand{\eps}{{\varepsilon}}

\newcommand{\abs}[1]{\left|{#1}\right|}
\newcommand{\norm}[1]{\lVert{#1}\rVert}

\newcommand{\F}{{\mathbb{F}}}
\newcommand{\N}{{\mathbb{N}}}

\newcommand{\R}{{\mathbb{R}}}

\newcommand{\Mean}{{{\mathbb{E}}}}
\newcommand{\Prob}{{{\mathbb{P}}}}

\newcommand{\cF}{{\mathcal{F}}}

\newcommand{\om}{{\omega}}

\begin{document}
\title[PPDE]
{Path-dependent Hamilton-Jacobi equations\\  with super-quadratic growth in the gradient
\\ and the vanishing viscosity method}

\author{Erhan Bayraktar}
\address
{Department of Mathematics, 
University of  Michigan,
Ann Arbor, MI 48109, United States}
\thanks{The research of the first author was
 supported in part by {\color{black} NSF-grant  DMS-2106556.
 The research of the second author was supported in part
 by NSF-grant DMS-2106077.}}
\email{erhan@umich.edu}
\author{ Christian Keller}
\address
{Department of Mathematics, 
University of  Central Florida,
Orlando, FL 32816, United States}
\email{christian.keller@ucf.edu}



\begin{abstract}
 The non-exponential Schilder-type theorem 
in Backhoff-Veraguas, Lacker  and Tangpi  
[Ann.~Appl.~Probab., 30 (2020), pp.~1321--1367]
is expressed as a convergence result for path-dependent partial differential equations
 with  appropriate notions of generalized solutions.
This entails a non-Markovian counterpart to the vanishing viscosity method.

  We show uniqueness of maximal subsolutions for
path-dependent viscous Hamilton-Jacobi equations  related to 
 convex super-quadratic backward stochastic differential
equations.  

 We establish well-posedness for the  Hamilton-Jacobi-Bellman equation associated to a Bolza problem of the calculus of variations with path-dependent
terminal cost. In particular, 
 uniqueness among lower semi-continuous solutions holds
and 
state constraints are admitted.
\end{abstract}
\maketitle
\pagestyle{plain}

\section{Introduction}

 Backhoff-Veraguas, Lacker  and Tangpi  \cite{BLT} derived a non-exponential Schilder-type theorem,
 which they used to obtain new limit theorems for backward stochastic differential equations (BSDEs)
  and, in the Markovian case,
 for the corresponding partial differential equations 
 (PDEs). They posed the question {\color{black} whether} it is possible to have a corresponding PDE result in the
 non-Markovian case as well.
Purpose of our work is to provide an answer to this question.

We establish well-posedness for (second order) path-dependent viscous Hamilton-Jacobi equations
and for (first-order) path-dependent Hamilton-Jacobi-Bellman (HJB) equations with possibly super-quadratic growth in
the gradient. Together with a modification of the Schilder type theorem in \cite{BLT}, we obtain 
a non-Markovian vanishing viscosity result
for path-dependent PDEs and thereby address the mentioned open problem in \cite{BLT}.

The notions of solutions for our path-dependent PDEs are in the spirit of contingent solutions for PDEs
(see, e.g., \cite{Frankowska89AMO}), also known as Dini solutions
(see, e.g., \cite{BardiCapuzzoDolcetta} and \cite{Vinter})
 or minimax solutions (see, e.g., \cite{SubbotinBook}).
 
 In the context of contingent or Dini solutions for first-order standard PDEs related to Bolza problems,  
 \cite{DM-F00}  and also \cite{PQ02}
 are very close to our approach.
More recent works in this direction are  
\cite{Misztela14}, \cite{BernisBettiol18COCV} and \cite{BernisBettiol20chapter}.
Regarding the possible use of viscosity solution techniques,  we refer the reader to the remarks on p.~1202 in
\cite{BettiolVinter17SICON}.  In particular, fast growth in the gradient, discontinuity of the Lagrangian,
 and extended real-valued 
lower semi-continuous
 terminal data (to allow right-end point constraints
 in optimal control problems) cause non-trivial issues. For example, solutions of HJB equations can 
 {\color{black} be}
  expected then to be
 only lower semi-continuous. 

 In the second-order case, the only works we are aware of that use
 contingent-type solutions are \cite{SubbotinaEtAl85} and \cite{Subbotina06},
 where Isaacs equations corresponding to Markovian stochastic differential games with drift control and
 bounded control spaces are studied. However, similar constructions (stochastic or Gaussian derivatives)
 are also used in \cite{Haussmann92MOR} and \cite{Haussmann94SICON}.

In the context of first-order path-dependent PDEs, \cite{KaiseEtAl18PPDE} is most relevant. It deals
with a  calculus of variations problem involving a path-dependent terminal cost and the  related path-dependent 
HJB equation.
The setting is very close to our problem (DOC) below. The main difference is that in \cite{KaiseEtAl18PPDE}
the terminal cost is required to be Lipschitz continuous, which leads to Lipschitz continuity of the corresponding value function
and also makes it possible in \cite{KaiseEtAl18PPDE}
 to develop a viscosity solution theory. In our work, we require only continuity resp.~lower
semi-continuity of the terminal cost, which is one of the reasons why we establish a Dini resp.~minimax solution theory.
For the current state of the art for first-order path-dependent PDEs
and for further relevant works,
see \cite{GLP21AMO} and the references therein.

In the literature 
\cite{EKTZ11,K14,ETZ_I,ETZ_II,PhamZhang14SICON,EkrenZhang16PUQR, RTZ14overview,
Ren16AAP, Ren17Stoch,Ekren17SPA,RTZ17,cosso18,
RTZ20SICON,RR20SIMA} 
on viscosity solutions of second-order path-dependent PDEs,
the Hamiltonian is required to grow at most linearly in the gradient
(the same condition is also needed in \cite{CossoRusso19Osaka}, where a notion of strong-viscosity
solutions is used). Overcoming this restriction for any notion of generalized solutions
 has been a longstanding open problem.
By proving wellposedness of maximal (Dini) subsolutions for a class
of second-order path-dependent PDEs with quadratic and even super-quadratic growth in the gradient, 
we establish first results related to this problem.

Non-Markovian large deviation problems and their connections to path-dependent PDEs are also
studied in \cite{MRTZ16}. In contrast to our work, in \cite{MRTZ16} only the (limiting) rate function is characterized as a
solution of a (first-order) path-dependent PDE. 
Moreover, the terminal condition  
is required to be Lipschitz continuous
whereas we need only continuity.

\section{Setup}
\subsection{Notation}
Let $\Omega=C([0,T],\R^d)$. The canonical process on $\Omega$ is denoted by $X$, i.e.,
$X(t,\omega)=\omega(t)$ for each $(t,\omega)\in [0,T]\times\Omega$.
Let $\F=\{\cF_t\}_{t\in [0,T]}$ be the (raw) filtration generated by $X$.
Given a probability measure $P$ on $(\Omega,\cF_T)$, denote by
$\F^P=\{\cF_t^P\}_{t\in [0,T]}$ the $P$-completion of the right-limit of $\F$.

We equip $\Omega$ with the  supremum norm $\norm{\cdot}_\infty$ and $[0,T]\times\Omega$ 
with the pseudo-metric $\mathbf{d}_\infty$
defined by  
\begin{align*}
\mathbf{d}_\infty((t_1,\omega_1),(t_2,\omega_2)):=\abs{t_1-t_2}+\sup\nolimits_{s\in [0,T]}
 \abs{\omega_1(s\wedge t_1)-\omega_2(s\wedge t_2)}.
\end{align*}
Continuity and semi-continuity of functions defined on $\Omega$ (resp.~$[0,T]\times\Omega$)
are to be understood with respect to $\norm{\cdot}_\infty$ (resp.~$\mathbf{d}_\infty$). Note that semi-continuous functions on
$[0,T]\times\Omega$ are $\F$-progressive. From now on, we write l.s.c.~(resp.~u.s.c.) instead
of lower semi-continuous (resp.~upper semi-continuous).

{\color{black}
With slight abuse of notation, we also use the notation $\norm{\cdot}_\infty$ to express the
sup-norm for functions belonging to other function spaces.
}

{\color{black} We often identify 
 vectors with constant functions, e.g., given a map 
$h:\Omega\to\R$, a vector $z\in\R^d$, and a path $\omega\in\Omega$, we write $h(\omega+z)$ instead of $h(\omega+z\,\bfone_{[0,T]})$.}

Given $(t_0,x_0,n)\in [0,T]\times\Omega\times\N$, 
denote by  $P_{t_0,x_0,n}$ be the unique probability measure on $(\Omega,\cF_T)$ such that
$\sqrt{n} X\vert_{[t_0,T]}$ is a $d$-dimensional standard $(P_{t_0,x_0,n},\F)$-Wiener process  
starting at $x_0(t_0)$
and
that $P_{t_0,x_0,n}(X\vert_{[0,t_0]}=x_0\vert_{{\color{black} [0,t_0]}})=1$.
We write $\Mean_{t_0,x_0,n}$ for the corresponding expected value.
Moreover, $P_{t_0,x_0}:=P_{t_0,x_0,1}$ and
$\F^{t_0,x_0,n}:=\F^{P_{t_0,x_0,n}}$.

 As space of controls,  the set $\mathcal{L}_b$ of all  bounded
$\F$-progressive processes from $[0,T]\times\Omega$ to $\R^d$ is used 
(whereas in \cite{BLT} the controls are $\F^{P_{0,0}}$-progressive).

We denote by $\mathrm{dom}$  the effective domain of an extended real-valued function.

\subsection{Data} 
Let $h: \Omega\to \R\cup\{\infty\}$ and $\ell:[0,T]\times\R^d\to\R\cup\{\infty\}$ be measurable
functions.
We  use the following hypotheses for $\ell$. 

(H1) The function $\ell=\ell(t,a)$ 
satisfies the Tonelli-Nagumo condition,
i.e., there is a function $\phi:[0,\infty)\to\R$  bounded from below
with $\phi(r)/r\to\infty$ as $r\to\infty$ such that $\ell(t,a)\ge \phi(\abs{a})$ on $[0,T]\times\R^d$.
Moreover, $\ell(t,\cdot)$ is l.s.c., proper, and convex for every $t\in [0,T]$. 

(H2) $\int_0^T \ell(t,0)\,dt<\infty$.

These hypotheses are nearly identical with the corresponding condition (TI) for the Lagrangian in \cite{BLT}
(where it is denoted by $g$). In some of our main results, we will, in addition to (H1) and (H2), also assume that
$\ell$ is continuous and finite-valued. In those cases, (TI) is satisfied as pointed out in \cite{BLT}.

\subsection{The optimal control problems and HJB equations}\label{SS:OCP_HJB}
Let $n\in\N$. The value for our stochastic optimal control problem $(\text{SOC}_n)$ with initial 
data $(t_0,x_0)\in [0,T]\times \Omega$ is given by
\begin{align*}
v_n(t_0,x_0)
:=\inf_{a\in\mathcal{L}_b^{t_0}}  \Mean_{t_0,x_0,n}\left[
\int_{t_0}^T \ell(t,a(t))\,dt +h(X+A^a)
\right],
\end{align*}
where  $\mathcal{L}_b^{t_0}=\{a\in\mathcal{L}_b: a\vert_{[0,t_0)}=0\}$ 
and 
$A^a$ is a continuous process on $[0,T]\times\Omega$ defined by $A^a(t):=\int_0^t a(s)\,ds$. 
The terminal value problem involving the corresponding HJB equation is 
\renewcommand{\theequation}
{\textrm{TVP}~$n$}
\begin{equation}\label{E:TVPn}
\begin{split}
\left(\partial_t +\frac{1}{2n} \partial_{xx}\right) u(t,x)+\inf_{a\in\R^d}\left[
a\cdot\partial_x u(t,x)+\ell(t,a)
\right]&=0\quad\text{in $[0,T)\times\Omega$},\\
u(T,x)&=h(x)\quad\text{on $\Omega$.}
\end{split}
\end{equation}
\begin{remark}\label{R:UniformBoundedness}
If (H1) and (H2) hold and $h$ is bounded, then 
$\sup_{n\ge 1} \norm{v_n}_\infty<\infty$.
\end{remark}
\setcounter{equation}{0}
\renewcommand{\theequation}{\thesection.\arabic{equation}}
The value for our deterministic optimal control problem $(\text{DOC})$ with initial 
data $(t_0,x_0)\in [0,T]\times \Omega$ is given by
\begin{align*}
v_0(t_0,x_0):=\inf_{x\in\mathcal{X}^{1,1}(t_0,x_0)} \left[
\int_{t_0}^T \ell(t,x^\prime(t))\,dt +h(x)
\right],
\end{align*}
where
\begin{align*}
\mathcal{X}^{1,1}(t_0,x_0):=\left\{
x\in \Omega:\, x\vert_{[0,t_0]}=x_0\vert_{[0,t_0]}\text{ and } x\vert_{{\color{black}(t_0,T)}}\in W^{1,1}(t_0,T;\R^d)
\right\}.
\end{align*}
{\color{black}
Here, $W^{1,p}(t_0,T;\R^d)$, 
$p\in [1,\infty]$, is the Sobolev space of all 
 $x\in L^p(t_0,T;\R^d)$ that
have a weak 
derivative $x^\prime\in L^p(t_0,T;\R^d)$.}
The terminal value problem involving the corresponding HJB equation is 
\renewcommand{\theequation}
{\textrm{TVP}}
\begin{equation}\label{E:TVP0}
\begin{split}
\partial_t u(t,x)+\inf_{a\in\R^d}\left[
a\cdot\partial_x u(t,x)+\ell(t,a)
\right]&=0\quad\text{in $[0,T)\times\Omega$},\\
u(T,x)&=h(x)\quad\text{on $\Omega$.}
\end{split}
\end{equation}
\setcounter{equation}{0}
\renewcommand{\theequation}{\thesection.\arabic{equation}}
\section{Notions of solutions of path-dependent HJB equations}
{\color{black} We call a function $u:[0,T]\times\Omega\to\R\cup\{\infty\}$ \emph{non-anticipating}
if $u(t,x)=u(t,x(\cdot\wedge t))$ for every $(t,x)\in [0,T]\times\Omega$.
Note that whenever a function on $[0,T]\times\Omega$ is l.s.c.~or u.s.c.~(with respect to $\mathbf{d}_\infty$),
then it is automatically non-anticipating.}

\subsection{Dini solutions} 
Given a {\color{black} non-anticipating} function $u:[0,T]\times\Omega\to\R\cup\{\infty\}$, we 
define the \emph{lower} and \emph{upper Dini derivative}
\begin{align*}
d_- u(t_0,x_0)(1,a)&:=\varliminf_{\delta\downarrow 0} \frac{u(t_0+\delta,x_0(\cdot\wedge t_0)
+{\color{black} A^a(\cdot\vee t_0)-A^a(t_0)})
-u(t_0,x_0)}{\delta}{\color{black},}\\
d_+ u(t_0,x_0)(1,a)&:=\varlimsup_{\delta\downarrow 0} \frac{u(t_0+\delta,x_0(\cdot\wedge t_0)
+{\color{black} A^a(\cdot\vee t_0)-A^a(t_0)})
-u(t_0,x_0)}{\delta}
\end{align*}
at  points $(t_0,x_0)\in [0,T)\times\Omega$ 
in direction $(1,a)\in\R\times\R^d$.
Here, in unison with the process $A^a$ in Subsection~\ref{SS:OCP_HJB}, $A^a(t)=at$.

The following (path-dependent) notion of Dini semi-solutions 
is motivated by the notion of (contingent) 
solutions used in Theorem~4.1 of \cite{DM-F00} for HJB equations related to Bolza problems.
Our notion is also related to the infinitesimal version of minimax solutions for path-dependent Isaacs equations  used in  \cite{Lukoyanov01JAMM}.

\begin{definition}\label{D:1stDiniSol}
Let $u:[0,T]\times \Omega\to \R\cup\{\infty\}$.

(i) We call $u$ a 
\emph{Dini supersolution}
of \eqref{E:TVP0}  if $u$ is l.s.c.,  $u(T,\cdot)\ge h$, and, for every 
$ (t_0,x_0)\in {\color{black} \mathrm{dom}(u)}$ {\color{black} with $t_0<T$}, 
\begin{align}\label{E:ContSupersolution}
\inf_{a\in\R^d}  \left[
d_-
u(t_0,x_0)(1,a)+
\ell(t_0,a)
\right]\le 0.
\end{align}

(ii)  We call $u$ a \emph{Dini subsolution}
of \eqref{E:TVP0} if  $u$ is 
u.s.c.,
 $u(T,\cdot)\le h$, and, for every 
$(t_0,x_0)\in\mathrm{dom}(u)$ with $t_0<T$,
\begin{align}\label{E:ContSubsolution}
\inf_{a\in\R^d} \left[d_+ u(t_0,x_0)(1,a)+\ell(t_0,a)\right]\ge 0.
\end{align}
We call $u$ a \emph{maximal Dini subsolution} of \eqref{E:TVP0} if $u$ is a 
Dini subsolution of \eqref{E:TVP0} and, for every Dini subsolution $v$ of \eqref{E:TVP0},
we have $u\ge v$.
\end{definition}

\begin{example}\label{E:Example}
Let $d=1$,  $h=\infty.\bfone_{K^c}$, where $K:=\{t\mapsto t^{1/2}\}\subset\Omega$, and $\ell$ be defined by
$\ell(t,a)=\abs{a}^{3/2}$.  Then 
$v_0$ satisfies
\begin{align*}
v_0(t_0,x_0)=\begin{cases}
\int_{t_0}^T 2^{-3/2}\, t^{-3/4}\,dt= 
2^{1/2}\left(T^{1/4}-t_0^{1/4}\right) &\text{if $x_0\vert_{[0,t_0]}(t)=t^{1/2}$,}\\
\infty &\text{otherwise.}
\end{cases}
\end{align*}
Note that, for each $(t_0,x_0)\in\mathrm{dom}(v_0)$ with $t_0<T$ and each $a\in\R$, we have 
$d_- v_{\color{black} 0}(t_0,x_0)(1,a)=\infty$ for which the constant perturbation involving $a$ is responsible.
(To obtain a finite value for our Dini derivative, we would need to permit the non-constant perturbation $t\mapsto t^{1/2}$.)
Thus 
$\inf_{a\in\R} [d_- v_{\color{black} 0}(t_0,x_0)(1,a)+\abs{a}^{3/2}]\le 0$ is never satisfied. Hence, in our example
there does not exist a Dini supersolution of \eqref{E:TVP0}.
This justifies the need for an appropriate weaker notion
of solution (see Subsection~\ref{SS:Minimax}).
\end{example}

Given {\color{black} a non-anticipating function} $u:[0,T]\times\Omega\to\R$, 
we define the
\emph{upper stochastic Dini derivative}  
\begin{align*}
&d_+^{1,2} u(t_0,x_0)(1,a,n^{-1} I_d)\\&\qquad:=
\varlimsup_{\delta\downarrow 0}\frac{
\Mean_{t_0,x_0,n} \left[u(t_0+\delta,X
+{\color{black} A^a(\cdot\vee t_0)-A^a(t_0)})
-u(t_0,x_0)\right]}{\delta}.
\end{align*}
at  points $(t_0,x_0)\in [0,T)\times\Omega$  in direction $(1,a,n^{-1} I_d)\in\R\times\R^d\times\R^{d\times d}$
(cf.~\cite{SubbotinaEtAl85, Subbotina06, Haussmann92MOR, Haussmann94SICON}).

The following notion of subsolutions for second order path-dependent PDEs 
is motivated by the minimax solutions used in \cite{SubbotinaEtAl85, Subbotina06} in a
Markovian framework.

\begin{definition}\label{D:2ndDiniSol}
Let $u:[0,T]\times\Omega\to\R$. 
We call $u$ a  \emph{Dini subsolution} of \eqref{E:TVPn} if $u$ is u.s.c.,
$u(T,\cdot)\le h$, and, 
for every $(t_0,x_0)\in [0,T)\times\Omega$,
\begin{align}\label{E:Subsolution2}
\inf_{a\in\R^d}\left[
d^{1,2}_+ u(t_0,x_0)(1,a,n^{-1}I_d)+\ell(t_0,a)
\right]
\ge 0.
\end{align}
We call $u$ a \emph{maximal Dini subsolution} of \eqref{E:TVPn} if $u$ is a 
Dini subsolution of \eqref{E:TVPn} and, for every Dini subsolution $v$ of \eqref{E:TVPn},
we have $u\ge v$.
\end{definition}

\begin{remark}
In our specific setting, the use of Dini type semiderivatives such as those introduced in  this section suffices.
This motivates us to call our generalized solutions Dini solutions. 
 For more general data,  path-dependent counterparts of contingent derivatives
such as the Clio derivatives in \cite{AubinHaddad02PPDE} need to be utilized.
Corresponding generalized solutions would be called contingent solutions.
 \end{remark}

\subsection{Minimax solutions}\label{SS:Minimax}
Here, we introduce a weaker notion of solution, which is an adjustment of the 
infinitesimal notion of minimax solutions in \cite{BK18JFA}.
It is also motivated by the notion of (l.s.c.~contingent) 
solutions  used in Theorem~5.1 of \cite{DM-F00}.
The problem in Example~\ref{E:Example}, which partially motivated this weaker notion,
 is overcome by allowing non-constant perturbations
(see also the notion of contingent solutions in \cite{Carja12SICON} that are defined via
contingent derivatives with function-valued directions).

\begin{definition}\label{D:minimax}
Let $u:[0,T]\times \Omega\to \R\cup\{\infty\}$ {\color{black} be non-anticipating}.

(i) We call $u$ a \emph{minimax supersolution} of \eqref{E:TVP0}  if $u$ is l.s.c.,   $u(T,\cdot)\ge h$, and, for every 
$ (t_0,x_0)\in {\color{black} \mathrm{dom}(u)}$ {\color{black} with $t_0<T$}, 
\begin{align}\label{E:Supersolution}
\inf_{x\in \mathcal{X}^{1,1}(t_0,x_0)}
 \varliminf_{\delta\downarrow 0} \left[ u(t_0+\delta,x)-u(t_0,x_0)+
\int_{t_0}^{t_0+\delta} \ell(s,x^\prime(s))\,ds
\right]\delta^{-1}\le 0.
\end{align}

(ii)  We call $u$ an \emph{l.s.c.~minimax
subsolution} of \eqref{E:TVP0} if  $u$ is l.s.c.,   $u(T,\cdot)\le h$, and, for every 
$(t_0,x_0)\in [0,T)\times\Omega$ and every
$(t,x)\in \left( (t_0,T]\times\mathcal{X}^{1,1}(t_0,x_0)\right)\cap\mathrm{dom}(u)$
with $\int_{t_0}^t \ell(s,x^\prime(s))\,ds<\infty$,
\begin{align}\label{E:Subsolution}
 \varliminf_{\delta\downarrow 0}
\left[
u(t-\delta,x)-u(t,x)-
\int_{t-\delta}^{t} \ell(s,x^\prime(s))\,ds
\right]\delta^{-1} \le 0.
\end{align}

(iii)  We call $u$  an \emph{l.s.c.~minimax solution} of \eqref{E:TVP0} if $u$ is  a minimax super- and an
l.s.c.~minimax subsolution of \eqref{E:TVP0}.
\end{definition}

{\color{black} \subsection{Consistency with classical solutions}
First, we provide the definitions for path derivatives.
The first-order ones are due to Kim \cite{KimBook} and  the second-order ones are due to Dupire \cite{dupirefunctional}.
Our presentation follows \cite{ETZ_I} and \cite{Lukoyanov03} .

\begin{definition}
Let $u:[0,T]\times\Omega\to\R$.

(i) We write $u\in C^{1,1}([0,T]\times\Omega)$ if $u\in C([0,T]\times\Omega,\R)$ and if there
are functions $\partial_t u\in C([0,T]\times\Omega,\R)$ and $\partial_x u\in C([0,T]\times\Omega,\R^d)$ 
called \emph{first-order path derivatives}
such
that, for every  $(t_0,x_0)\in [0,T)\times\Omega$, every $x\in\mathcal{X}^{1,1}(t_0,x_0)$ and
every $t\in (t_0,T]$, we have
\begin{align*}
u(t,x)-u(t_0,x_0)=\int_{t_0}^t \left[\partial_t u(s,x)+x^\prime(s)\cdot \partial_x u(s,x)\right]\,ds.
\end{align*}

(ii) We write $u\in C^{1,2}([0,T]\times\Omega)$ if $u\in C^{1,1}([0,T]\times\Omega)$
with corresponding first-order path derivatives $\partial_t u$ and $\partial_x u$ and if there is
a function $\partial_{xx} u\in C([0,T]\times\Omega,\R^{d\times d})$ called
\emph{second-order path derivative} such that, for every $(t_0,x_0)\in [0,T)\times\Omega$,
every probability measure $P$ on $\cF_T$ such that $X$ is a $(P,\F)$-It\^o-semimartingale
after time $t_0$ with bounded characteristics 
and with $P(X\vert_{[0,t_0]}=x_0\vert_{[0,t_0]})=1$, and every
$t\in (t_0,T]$, we have
\begin{align*}
u(t,X)-u(t_0,x_0)&=\int_{t_0}^t \partial_t u(s,X)\,ds+\int_{t_0}^t \partial_x u(s,X)\cdot dX(s)\\
&\qquad\qquad+
\int_{t_0}^t  \frac{1}{2}\partial_{xx} u(s,X):d\langle X(s)\rangle,\quad\text{$P$-a.s.}
\end{align*}
Here, $\langle X(\cdot)\rangle$ is the quadratic variation of $X\vert_{[t_0,T]\times\Omega}$
and, given matrices $M$, $N\in\R^{d\times d}$, $M:N$ is the trace of $MN^\top$.
\end{definition}

\begin{remark}
If $u\in C^{1,1}([0,T]\times\Omega)$, then its first-order path derivatives are
uniquely determined. If, in addition, $u\in C^{1,2}([0,T]\times\Omega)$, then
its second-order path-derivative is uniquely determined as well.
We refer to Section~2.3 of \cite{ETZ_I} for more details.
\end{remark}

\begin{definition}\label{D:ClassicalSol}
Let $u:[0,T]\times\Omega\to\R$.

(i) We call $u$ a \emph{classical subsolution} (resp.~\emph{classical supersolution}, 
\emph{classical solution}) of \eqref{E:TVP0}  if $u\in C^{1,1}([0,T]\times\Omega)$,
$u(T,\cdot)\le h$ (resp.~$u(T,\cdot)\ge h$, $u(T,\cdot)=h$),
and, for every $(t,x)\in [0,T)\times\Omega$,
\begin{align*}
\partial_t u(t,x)+\inf_{a\in\R^d}\left[
a\cdot\partial_x u(t,x)+\ell(t,a)
\right] \ge\,\text{(resp.~$\le$, $=$)}\,0.
\end{align*}
 
(ii) We call $u$ a \emph{classical subsolution} (resp.~\emph{classical supersolution}, 
\emph{classical solution}) of \eqref{E:TVPn}  if $u\in C^{1,2}([0,T]\times\Omega)$ 
 (resp.~$u(T,\cdot)\ge h$, $u(T,\cdot)=h$) and,
  for every $(t,x)\in [0,T)\times\Omega$,
  \begin{align*}
\partial_t u(t,x)+\frac{1}{2n}\partial_{xx} u(t,x)+\inf_{a\in\R^d}\left[
a\cdot\partial_x u(t,x)+\ell(t,a)
\right] \ge\,\text{(resp.~$\le$, $=$)}\,0.
\end{align*}
\end{definition}

The following result follows immediately from Definitions~\ref{D:1stDiniSol}, \ref{D:2ndDiniSol},
and~\ref{D:ClassicalSol}.

\begin{proposition}[Consistency of Dini solutions with classical solutions]
Assume that $\ell$ is continuous.

(i) Let $u\in C^{1,1}([0,T]\times\Omega)$. Then $u$ is a 
classical subsolution (resp.~classical supersolution, classical solution)
of \eqref{E:TVP0} 
if and only if $u$ is  a Dini subsolution (resp.~Dini supersolution, Dini solution
of \eqref{E:TVP0}.

(ii) Let $u\in C^{1,2}([0,T]\times\Omega)$. Then
$u$ is a classical subsolution of \eqref{E:TVPn} if and only if
$u$ is a Dini subsolution of \eqref{E:TVPn}.
\end{proposition}


\begin{proposition}[Partial consistency of l.s.c.~minimax solutions with classical solutions]\label{P:ConsMinimax}
Assume that $\ell$ is continuous and real-valued.
Let $u\in C^{1,1}([0,T]\times\Omega)$. 

(a) If
$u$  is an l.s.c.~minimax subsolution  of \eqref{E:TVP0}, then $u$ is a classical subsolution of
\eqref{E:TVP0}.

(b) If $u$ is a classical supersolution of \eqref{E:TVP0}, then $u$ is a minimax supersolution of
\eqref{E:TVP0}.
\end{proposition}

\begin{remark}
The converse of Proposition~\ref{P:ConsMinimax}~(b) cannot expected to be valid in general because
the infimum over $\mathcal{X}^{1,1}(t_0,x_0)$ in \eqref{E:Supersolution} can be strictly less than 
the corresponding infimum over $\R^d$. For similar reasons, we cannot expect the converse of 
 Proposition~\ref{P:ConsMinimax}~(a) to be valid in general.
\end{remark}

\begin{proof}[Proof of Proposition~\ref{P:ConsMinimax}]
Part~(b) follows immediately from Definition~\ref{D:minimax}~(i) and Definition~\ref{D:ClassicalSol}~(i).
It remains to prove part~(a). To this end, 
fix $(t_0,x_0)\in [0,T)\times\Omega$
and assume that $u$ is an l.s.c.~minimax subsolution of \eqref{E:TVP0}.
Fix $t\in (t_0,T]$.
Fix $a\in \R^d$.
Then, for any $x(\cdot)\in\mathcal{X}^{1,1}(t_0,x_0)$ with a continuous derivative $x^\prime$ that satisfies $x^\prime(t)=a$,
we have
\begin{align*}
0&\ge \varliminf_{\delta\downarrow 0}
\left[
u(t-\delta,x)-u(t,x)-
\int_{t-\delta}^{t} \ell(s,x^\prime(s))\,ds
\right]\delta^{-1}\\ &=
 \varliminf_{\delta\downarrow 0}
\left[
-\left(
\int_{t-\delta}^t \partial_t u(s,x)+\partial_x u(s,x)\cdot x^\prime(s)\,ds
\right)-
\int_{t-\delta}^{t} \ell(s,x^\prime(s))\,ds
\right]\delta^{-1}\\
&=-\partial_t u(t,x)-\partial_x u(t,x)\cdot a-\ell(t,a).
\end{align*}
Since $u\in C^{1,1}([0,T]\times\Omega)$, $\ell$ is continuous, and  $a$ was arbitrary in $\R^d$, we have
\begin{align*}
\inf_{a\in\R^d}\left[ \partial_t u(t_0,x_0)+\partial_x u(t_0,x_0)\cdot a+\ell(t_0,a)\right]\ge 0,
\end{align*}
i.e., $u$ is classical subsolution of \eqref{E:TVP0}. 
\end{proof}
}
\section{Main results}

\begin{theorem}\label{T:MainResult}
Assume (H1)  {\color{black}with $\phi(r)=\abs{r}^p$ for some $p>1$} and (H2). 
Let $h$ be  continuous and bounded. Then  $(v_n)$ converges to $v_0$ uniformly on compacta and
$v_0$ is continuous.
 Moreover, $v_0$ is the unique l.s.c.~minimax solution 
of \eqref{E:TVP0} that is bounded from below.
If, in addition, $\ell$ is  
 continuous and finite-valued, 
then we have the following:

(i)  For each $n\in\N$, the function  $v_n$ is the unique bounded maximal 
Dini subsolution of \eqref{E:TVPn}

(ii) The function $v_0$ is the unique bounded maximal   
Dini subsolution of \eqref{E:TVP0}. 
\end{theorem}

\begin{proof}
The uniform convergence of $(v_n)$  to $v_0$  on compacta and the continuity of  $v_0$ are
proven in Section~\ref{S:Convergence}.  Theorem~\ref{T:1stOrder} 
characterizes $v_0$ as the unique minimax 
solution resp.~as  the unique bounded
maximal 
Dini subsolution of \eqref{E:TVP0}. 
Theorem~\ref{T:TVPn} addresses the remaining part, i.e., wellposedness of \eqref{E:TVPn}.
\end{proof}

\begin{remark}
Well-posedness of \eqref{E:TVPn} requires $h$ to be only u.s.c.~and bounded (see Theorem~\ref{T:TVPn}).
The corresponding result in the Markovian case treated in \cite{BLT} is of similar strength
(well-posedness holds for maximal viscosity supersolutions of 
the corresponding viscous Hamilton-Jacobi equations, which is due to \cite{DrapeauMainberger16EJP}).

\end{remark}

\begin{theorem}\label{T:1stOrder}
Assume (H1).

(a)  Let $h$ be l.s.c., proper, and bounded from below. Then the value function $v_0$ is the  unique l.s.c.~minimax solution of \eqref{E:TVP0} that is bounded from below.

(b) Assume 
(H2). 
 Let $\ell$ be continuous and finite-valued. Let
 $h$ be  u.s.c.~and 
 bounded.
 Then  $v_0$ is the unique maximal 
bounded Dini subsolution of \eqref{E:TVP0}.
\end{theorem}

\begin{proof}
See Section~\ref{S:1stOrder}.
\end{proof}

\section{Proof of the convergence result}\label{S:Convergence}
Consider the  semicontinuous envelopes $v_\ast$ and $v^\ast$ defined by
\begin{align*}
v_\ast(t_0,x_0)&:=
\sup\limits_{\substack{\delta> 0,\\ n\in\N}}\quad
 \inf\limits_{\substack{(t,x)\in O_\delta(t_0,x_0),\\m\ge n}}\quad v_m(t,x),\\
 v^\ast(t_0,x_0)&:=
\inf\limits_{\substack{\delta> 0,\\ n\in\N}}\quad
 \sup\limits_{\substack{(t,x)\in O_\delta(t_0,x_0),\\m\ge n}}\quad v_m(t,x)
\end{align*}
for every $(t_0,x_0)\in [0,T]\times \Omega$. Here, $O_\delta(t_0,x_0)$ is the open $\delta$-neighborhood of
$(t_0,x_0)$ in $([0,T]\times \Omega,\mathbf{d}_\infty)$.

First, we establish an auxiliary result.
\begin{lemma}\label{L:Closure}
Assume (H1). Let $t_n\to t_0$ in $[0,T]$ as $n\to\infty$.
Consider a probability space $(\tilde{\Omega},\tilde{\cF},\tilde{P})$.
Let $(a_n)$ be a sequence in $L^1=L^1([0,T]\times\tilde{\Omega},dt\otimes d\tilde{P};\R^d)$
that converges weakly to some $a\in L^1$.
 Then
$\Mean^{\tilde{P}} \int_{t_0}^T \ell(t,a(t))\,dt\le
\varliminf_n \Mean^{\tilde{P}}\int_{t_n}^T \ell(t,a_n(t))\,dt$.
\end{lemma}

\begin{proof}
We follow the lines of the proof of the corresponding deterministic closure theorem
10.8.ii in \cite{Cesari}.
First, note that, by lower semi-continuity and convexity of $\ell$,
for every $\delta>0$, we have
$\Mean^{\tilde{P}} \int_{t_0+\delta}^T \ell(t,a(t))\,dt\le\varliminf_n 
\Mean^{\tilde{P}} \int_{t_0+\delta}^T \ell(t,a_n(t))\,dt$ (more details can be found the proof of Lemma~A.1 of \cite{BLT}).
Next, fix $\eps>0$. 
Since $\ell$ is bounded from below,
there is a $\delta_0>0$ independent from $n$ such that, for all $\delta\in (0,\delta_0)$,
 we have
$\Mean^{\tilde{P}} \int_{t_n}^{t_0+\delta} \ell(t,a_n(t))\,dt 
>-\eps$ 
and thus
$\varliminf_n \Mean^{\tilde{P}} \int_{t_n}^T \ell(t,a_n(t))\,dt +\eps\ge
\Mean^{\tilde{P}} \int_{t_0+\delta}^T \ell(t,a(t))\,dt$. 
Again, as $\ell$ is bounded from below,  either
the right-hand side of 
the previous inequality
converges to $\Mean^{\tilde{P}} \int_{t_0}^T \ell(t,a(t))\,dt$
as $\delta\to 0$ or otherwise the left-hand side equals $\infty$.
This concludes the proof as $\eps$ was chosen arbitrarily.
\end{proof}

The following two statements  adapt
Theorem~2.2 in \cite{BLT} to our slightly more general setting. 

\begin{lemma}\label{L:LSC}
Assume (H1)  {\color{black}with $\phi(r)=\abs{r}^p$ for some $p>1$} and (H2).  
Let $h$ be l.s.c.~and bounded from below. Then $v_0\le v_\ast$.
\end{lemma}

\begin{proof}
Let $(t_0,x_0)\in [0,T]\times \Omega$. 
It suffices to consider the case $v_\ast(t_0,x_0)<\infty$.
Let $(t_n,x_n)_n$ be a sequence in $[0,T]\times \Omega$ that converges to $(t_0,x_0)$ in $\mathbf{d}_\infty$
and that satisfies $v_\ast (t_0,x_0)=\varliminf_n v_n(t_n,x_n)$.
Let $(a_n)$ be a sequence in $\mathcal{L}_b$ such that each $a_n$ 
belongs to $\mathcal{L}_b^{t_n}$ and 
is an $n^{-1}$-minimizer of $(\text{SOC}_n)$ with initial
data $(t_n,x_n)$.
 Then there exists a  subsequence  $(\theta_k)_k:=(t_{n_k},x_{n_k},n_k)_k$ of $(t_n,x_n,n)_n$ with
\begin{align*}
v_\ast(t_0,x_0)&=\lim_k \Mean_{\theta_k} \left[
\int^T_{t_{n_k}} \ell(t,a_{n_k}(t))\,dt+h(X+A^{a_{n_k}})
\right]
\end{align*}
and, for all $k\in\N$,
\begin{align*}
v_\ast(t_0,x_0)-1\le \Mean_{\theta_k}\int_{t_{n_k}}^T\ell(t,a_{n_k}(t))\,dt+
\Mean_{\theta_k}\, h(X+A^{a_{n_k}})\le v_\ast(t_0,x_0)+1.
\end{align*}
Since $\ell$ and $h$ are bounded from below, we can assume that
\begin{align*}
v^1=\lim_k \Mean_{\theta_k} \int_{t_{n_k}}^T\ell(t,a_{n_k}(t))\,dt\text{ and }
v^2=\lim_k \Mean_{\theta_k}\, h(X+A^{a_{n_k}})
\end{align*}
for some $v^1$, $v^2\in\R$ with $v_\ast(t_0,x_0)=v^1+v^2$
(cf.~Theorem~11.1.i and its proof in \cite{Cesari}).
As $\sup_k  \Mean_{\theta_k}\int_{t_{n_k}}^T\ell(t,a_{n_k}(t))\,dt<\infty$ and 
$a_{n_k}\vert_{[0,t_{n_k}]}=0$ for all $k\in\N$,
one  can  {\color{black} proceed} 
nearly exactly as in the proof of Lemma~A.1 in \cite{BLT}
{\color{black} to show that the probability measures $(P_{\theta_k}\circ (A^{a_{n_k}})^{-1})_k$}
are tight. {\color{black} Let us point out the differences to \cite{BLT}. Thanks to our additional requirement that
the function $\phi$ from Hypothesis (H1) satisfies
$\phi(r)=\abs{r}^p$ for some $p>1$, we can estimate $\int_0^T \abs{a_{n_k}(t)}^p\,dt$ instead of 
$\int_0^T \abs{a_{n_k}(t)}\,dt$ (cf.~with the first displayed equation in the proof of Lemma~A.1 in \cite{BLT})
and thus we can invoke Lemma~2 in \cite{Zheng85} to obtain tightness
(cf.~also with the proof of Lemma~3.13 in \cite{TanTouzi13})
 instead of using the Aldous tightness criterion (Theorem~16.11 in \cite{Kallenberg2nd}).}
Let us also note that the sequence  $(P_{\theta_k})$ is weakly convergent
because for each $\eta\in C_b(\Omega)$,
\begin{align*}
\Mean_{\theta_k}\, \eta=\Mean_{0,0}\, \eta\left(x_{n_k}(\cdot\wedge t_{n_k})
+ \frac{1}{\sqrt{n_k}}\, 
{\color{black}(X-X(t_{n_k}))\bfone_{[t_{n_k},T]}}
\right) \to\eta(x_0(\cdot\wedge t_0))
\end{align*}
as $k\to\infty$, i.e., a sequence of copies of $X$ converges in distribution to the constant
$x_0(\cdot\wedge t_0)$.
{\color{black} Consequently, the probability measures $(P_{\theta_k}\circ(A^{a_{n_k}},X)^{-1})_k$ are tight.}
Thus, by Skorohod's representation theorem, there exists a probability space
$(\bar{\Omega},\bar{\cF},\bar{P})$ with a sequence of $\Omega\times\Omega$-valued random variables
$(\bar{A}_{n_k},\bar{X}_{n_k})_k$ that satisfies 
$\bar{P}\circ(\bar{A}_{n_k},\bar{X}_{n_k})^{-1}=P_{\theta_k}\circ(A^{a_{n_k}},X)^{-1}$ 
for each  ${\color{black} k}\in\N$
and that converges (after passing to a subsequence), $\bar{P}$-a.s., to some random variable $(\bar{A}_0,\bar{X}_0)$.
Next, define a sequence $(\bar{a}_k)$ of 
$\R^d$-valued processes on $[0,T]\times\bar{\Omega}$ by  
$\bar{a}_k(t):=a_{n_k}(t,\bar{X}_{n_k})$. 
Again as in the proof of Lemma~A.1 in \cite{BLT}, one can deduce that
$(\bar{a}_k)$ is equiabsolutely integrable and
thus  has a subsequential weak limit  in $L^1([0,T]\times\bar{\Omega},dt\otimes d\bar{P};\R^d)$
that we denote by $\bar{a}_0$ and that satisfies, 
by Lemma~\ref{L:Closure},
\begin{align*}
\Mean^{\bar{P}}
\int_{t_0}^T \ell(t,\bar{a}_0(t))\,dt
\le\varliminf_i
\Mean^{\bar{P}}
\int_{t_{n_{k_i}}}^T \ell(t,\bar{a}_{k_i}(t))\,dt
=
\varliminf_i
\Mean_{\theta_{k_i}}
\int_{t_{n_{k_i}}}^T \ell(t,a_{n_{k_i}}(t))\,dt
\end{align*}
as well as   $\bar{A}_0(t)=\int_0^t \bar{a}_0(s)\,ds$, $\bar{P}$-a.s., for every $t\in [0,T]$. Moreover,
$\bar{A}_0\vert_{[0,t_0]}=0$, $\bar{P}$-a.s.
Hence, 
together with $h$ being l.s.c.~and $\bar{X}_0$=$x_0(\cdot\wedge t_0)$, $\bar{\Prob}$-a.s., we have
\begin{align}\label{E:LSC:Last}
v_\ast(t_0,x_0)=v^1+v^2\ge \Mean^{\bar{P}}\left[\int_{t_0}^T \ell(t,\bar{a}_0(t))\,dt
+h(x_0(\cdot\wedge t_0)+\bar{A}_0)\right].
\end{align}
To conclude the proof, it suffices to note that there exists some $x\in\mathcal{X}^{1,1}(t_0,x_0)$
such that the right-hand side \eqref{E:LSC:Last} is greater than or equal to
$\int_{t_0}^T \ell(t,x^\prime(t))\,dt+h(x$)
(cf.~also Remark 2.6 of \cite{HaussmannLepeltier90SICON}).
\end{proof}

\begin{lemma}\label{L:USC}
Assume (H1) and (H2). 
Let $h$ be u.s.c.~and bounded. Then $v^\ast\le v_0$.
\end{lemma}

\begin{proof}
It suffices to follow the arguments of the first paragraph of the proof of Theorem 2.2 in \cite{BLT}
and make the obvious adjustments. 
For the convenience of the reader, we quickly go over it. Fix $(t_0,x_0)\in [0,T]\times\Omega$
and 
{\color{black}$a\in L^1(0,T;R^d)$ }
with $\int_0^T \ell(t,a(t))\,dt<\infty$. Given $N\in\N$, define
$a^N(t):=a(t)$ for $\abs{a(t)}\le N$ and $a^N(t):=(N/\abs{a(t)})\,a(t)$ for $\abs{a(t)}>N$.
Next, consider a sequence $(t_n,x_n)$ that converges to $(t_0,x_0)$ in $[0,T]\times\Omega$
and  satisfies $v^\ast(t_0,x_0)=
{\color{black}\varlimsup_n}
v_n(t_n,x_n)$. Then
\begin{align*}
&v^\ast(t_0,x_0)\le\varlimsup_n\Mean_{t_n,x_n,n}\left[
\int_{t_n}^T \ell(t,a^N(t))\,dt+h(X+A^{a^N}
{\color{black}(\cdot\vee t_n)}
-A^{a^N}(t_n))\right]\\
&\le\varlimsup_n  \int_{t_n}^T \ell(t,a^N(t))\,dt\\
&\, +{\color{black} \varlimsup_n \Mean_{0,0} \, h\left(x_{n}(\cdot\wedge t_{n})
+ \frac{1}{\sqrt{n}}\, 
{\color{black}(X-X(t_{n})\bfone_{[t_n,T]}})
+A^{a^N}
{\color{black}(\cdot\vee t_n)}
-A^{a^N}(t_n)\right)}\\
&\le \int_{t_0}^T \ell(t,a^N(t))\,dt +h(x_0(\cdot\wedge t_0)+A^{a^N}
{\color{black}(\cdot\vee t_0)}
-A^{a^N}(t_0))
\end{align*}
as $h$ is u.s.c.~as well as bounded and $\int_0^T \ell(t,a^N(t))\,dt<\infty$, which follows from $\int_0^T
\ell (t,a(t))\,dt<\infty$, (H2), and the convexity of $\ell(t,\cdot)$ (cf.~the first displayed equation after
(38) in \cite{BLT}). Finally, letting $N\to\infty$ concludes the proof  
as in \cite{BLT}.
\end{proof}

\begin{remark} [No Lavrentiev phenomenon]\label{R:NoLavrentiev}
Assume (H1) and (H2). 
Let $h$ be continuous and bounded. Then
\begin{align*}
v_0(t_0,x_0)=\inf_{x\in\mathcal{X}^{1,\infty}(t_0,x_0)} \left[
\int_{t_0}^T \ell(t,x^\prime(t))\,dt +h(x)
\right],
\end{align*}
where
\begin{align*}
\mathcal{X}^{1,\infty}(t_0,x_0):=\left\{
x\in \Omega:\, x\vert_{[0,t_0]}=x_0\vert_{[0,t_0]}\text{ and } x\vert_{[t_0,T]}\in W^{1,\infty}(t_0,T;\R^d)
\right\}.
\end{align*}
This result follows from the proofs  of Lemmata~\ref{L:LSC} and \ref{L:USC} but with
each $P_{t_0,x_0,n}$, $n\in\N$, replaced by the unique probability measure  
under which $X=x(\cdot\wedge t_0)$ a.s.
For a more direct proof, it suffices to slightly modify the proof of Proposition~4.1 in \cite{ButtazzoBelloni95}
(this  result is due to \cite{DeArcangelis89}), where the case $h\equiv 0$ is treated. 
\end{remark}

\begin{proof}[Proof of the first conclusion of Theorem~\ref{T:MainResult}]
By Lemmata~\ref{L:LSC} and \ref{L:USC}, $v_0\le v_\ast\le v^\ast\le v_0$. Thus $(v_n)$ converges to $v_0$ uniformly on compacta and $v_0$ is continuous
(cf.~Lemmata~V.1.5 and V.1.9 of \cite{BardiCapuzzoDolcetta} and keep Remark~\ref{R:UniformBoundedness} in mind).
\end{proof}

\section{Connections to BSDEs} \label{S:BSDE}
We present  parts of the theory of (convex) superquadratic BSDEs  from \cite{DrapeauEtAl13AOP}
that are relevant {\color{black} for 
 our work.}
Fix $(t_0,x_0)\in [0,T]\times\Omega$, $t_1\in [t_0,T]$, $n\in\N$, and a $\cF_{t_1}$-measurable random variable
$\xi:\Omega\to\R\cup\{\infty\}$. Consider the BSDE
\begin{equation}\label{E:BSDE}
\begin{split}
dY(t)&=-\sup_{a\in\R^d}[a\cdot Z(t)-\ell(t,a)]\,dt +Z(t)\,dX(t),
\,\text{on $[t_0,t_1]$, $P_{t_0,x_0,n}$-a.s.,}\\
Y(t_1)&=\xi.
\end{split}
\end{equation}

\begin{definition}
 A pair $(Y,Z)$ is a \emph{supersolution} of \eqref{E:BSDE} if
 \begin{itemize}
\item $Y:[t_0,t_1]\times\Omega\to\R$ is a c\`adl\`ag and $\F^{t_0,x_0,n}$-adapted process,
\item $Z:[t_0,t_1]\times\Omega\to\R^d$ is  an $\F^{t_0,x_0,n}$-predictable process
with $$\Mean_{t_0,x_0,n} \int_{t_0}^{t_1} \abs{Z(t)}^2\,dt<\infty,$$
\item   $(t,\omega)\mapsto \left[
\int_{t_0}^t Z(r)\,dX(r)
\right](\omega)$, $[t_0,t_1]\times\Omega\to\R$,
is an $(\F^{t_0,x_0,n},\Prob_{t_0,x_0,n})$-supermartingale,
and, 
\item for every $t$, $s\in [t_0,t_1]$ with $t\le s$, we have, $P_{t_0,x_0,n}$-a.s.,
\begin{align*}
Y(s)&\ge Y(t)+\int_t^s  \left\{-\sup_{a\in\R^d}[a\cdot Z(r)-\ell(r,a)]\right\}\,dr +\int_t^s Z(r)\,dX(r),\\
Y(t_1)&\ge \xi.
\end{align*}
\end{itemize}

 A pair $(Y,Z)$ is a \emph{minimal supersolution} of \eqref{E:BSDE} if
it is a supersolution of \eqref{E:BSDE} and, for every supersolution
$(\tilde{Y},\tilde{Z})$ of \eqref{E:BSDE}, we have $Y\le \tilde{Y}$, $dt\otimes dP_{t_0,x_0,n}$-a.e.
\end{definition}

\begin{proposition}\label{P:BSDE}
Assume (H1) and (H2). Let $\xi$ be bounded. Then  \eqref{E:BSDE} has a unique minimal supersolution
$(Y,Z)$ and we write
\begin{align}\label{E:BSDEvalue}
\mathcal{E}^{t_0,x_0,n}_{t,t_1}(\xi):=Y_t,\quad t\in [t_0,t_1].
\end{align}
\end{proposition}
\begin{proof}
By the proof of Theorem~3.4 in \cite{DrapeauEtAl16AIHP} and 
by Theorem~A.1 in \cite{DrapeauEtAl16AIHP},
the set of supersolutions of \eqref{E:BSDE} is non-empty.
Thus we can apply Theorem~4.17 in \cite{DrapeauEtAl13AOP},
which concludes the proof.
\end{proof}

From now on,  let $T=1$ in this section for the sake of simplicity. 
Fix $(t_0,x_0)\in [0,T]\times\Omega$. Recall 
that
$\mathcal{L}_b^{t_0}=\{a\in\mathcal{L}_b: a\vert_{[0,t_0)}=0\}$.
Given $a\in\mathcal{L}_b$, define $A^a:[0,T]\times\Omega\to\R^d$ by
$A^a(t,\om):=\int_0^t a(s,\om)\,ds$. 
We  also identify a control $a\in\mathcal{L}_b$ with the function
$\Omega\to L^0(\Omega,\cF_T)$, $\omega\mapsto a(\omega)$, where 
$[a(\omega)](t)=a(t,\omega)$.
Recall that $X(t,\om)=\om(t)$, $(t,\om)\in [0,T]\times\Omega$,
and that $\Mean_{t_0,x_0}=\Mean^{P_{t_0,x_0}}$.

For the next statement, we borrow the following constructions from \cite{BLT}.
Given $t\in [0,1)$,  $\omega_1$, $\omega_2\in\Omega$ with $\omega_2(0)=0$, a path
$\omega_1\odot_t\omega_2\in\Omega$ is defined by
\begin{align}\label{E:odot}
(\omega_1\odot_t \omega_2)(s):=\omega_1(s\wedge t)
+\sqrt{1-t}\,\omega_2\left(\frac{s-t}{1-t}\right)\bfone_{[t,1]}(s)
\end{align}
(note that in \cite{BLT} $\otimes$ is used instead of $\odot$) and a path $\omega_1^{(t)}\in\Omega$
is defined by
\begin{align*}
\omega_1^{(t)}(s):=\frac{1}{\sqrt{1-t}}\,
[\omega_1(t+s(1-t))-\omega_1(t)].
\end{align*}
Moreover, we use the function $\ell^{(t_0)}:[0,1]\times\R^d\to\R\cup\{\infty\}$ defined by
\begin{align}\label{E:ellt0}
\ell^{(t_0)}(t,a):=(1-t_0)\,\ell\left(t_0+t(1-t_0),\frac{a}{\sqrt{1-t_0}}\right).
\end{align}
Also set $(\omega_1\odot_1 \omega_2):=\omega_1$,  $\omega_1^{(1)}\equiv 0$, and
$\ell^{(1)}\equiv 0$.

\begin{lemma}\label{L:Rescaling}
Assume (H1) and (H2). Let $h$ be bounded.
Then
\begin{equation}\label{E:Rescaling}
\begin{split}
{\color{black} -v_1(t_0,x_0)}
=
\sup_{\check{a}\in\mathcal{L}_{b}} \Mean_{0,0}\left[
(-h)(x_0\odot_{t_0} [X+A^{\check{a}(X)}])
-\int_0^1 \ell^{(t_0)} (t,\check{a}(t,X))\,dt
\right].
\end{split}
\end{equation}
\end{lemma}

The calculations in our  proof are essentially the same as those in the proof of Lemma~5.1 in \cite{BLT}.
Although of similar nature, the corresponding statements are different. This permits a 
more elementary proof in our case.

\begin{proof}[Proof of Lemma~\ref{L:Rescaling}]

One should keep in mind, that $X=x_0\odot_{t_0} X^{(t_0)}$, $P_{t_0,x_0}$-a.s.
We only   consider  the case $t_0<1$  as otherwise \eqref{E:Rescaling} is clearly satisfied.

 Step~1. Fix $a\in\mathcal{L}_b^{t_0}$. Note that
\begin{align*}
&\Mean_{t_0,x_0}\left[
\int_{t_0}^1 \ell(t,a({\color{black}t},x_0\odot_{t_0}X^{(t_0)}))\,dt
\right]=
\Mean_{0,0} \left[
\int_{t_0}^1 \ell(t,a({\color{black} t},x_0\odot_{t_0} X))\,dt
\right]\\
&=\Mean_{0,0}\left[
(1-t_0)\int_0^1 \ell(t_0+t(1-t_0),a(t_0+t(1-t_0),x_0\odot_{t_0} X))\,dt
\right]\\
&=\Mean_{0,0}\left[ 
\int_0^1\ell^{(t_0)} (t,\check{a}(t,X))\,dt
\right],
\end{align*}
where $\check{a}\in\mathcal{L}_b$ is defined by 
\begin{align*}
\check{a}(t,\omega):=\sqrt{1-t_0}\, a(t_0+t(1-t_0),x_0\odot_{t_0} \omega).
\end{align*}
Also note that
\begin{align*}
&\Mean_{t_0,x_0}\left[
(-h)(x_0\odot_{t_0}[
X+A^{a(x_0\odot_{t_0} X^{(t_0)}{\color{black})}}
]^{(t_0)})
\right]\\
&=\Mean_{0,0}\left[
(-h)(x_0\odot_{t_0}[X+(A^{a(x_0\odot_{t_0} X)})^{(t_0)}])
\right]\\
&=\Mean_{0,0}\left[
(-h)(x_0\odot_{t_0}[X+A^{\color{black}\check{a}}])
\right]
\end{align*}
because, for every $t\in [0,1]$, 
\begin{align*}
(A^{a(x_0\odot_{t_0} X)})^{(t_0)}{\color{black}(t)}=
\frac{1}{\sqrt{1-t_0}}\int_{t_0}^{t_0+t(1-t_0)} a(s,x_0\odot_{t_0} X)\,ds=\int_0^t \check{a}(s,X)\,ds.
\end{align*}
Consequently, the left-hand side of \eqref{E:Rescaling} is less than or equal to
the right-hand side of \eqref{E:Rescaling}.

Step~2. Fix $\check{a}\in\mathcal{L}_b$.
Define $a\in\mathcal{L}_b^{t_0}$ by
\begin{align*}
a(t,\omega):=\frac{1}{\sqrt{1-t_0}} \,
\check{a}\left(\frac{t-t_0}{1-t_0},\omega^{(t_0)}\right)\,\bfone_{[t_0,1]\times\{x_0\vert_{[0,t_0]}\}}
(t,\omega\vert_{[0,t_0]}).
\end{align*}
Going over the calculations in Step~1 backward, one can deduce  that the
right-hand side of \eqref{E:Rescaling} is less than or equal to
the left-hand side of \eqref{E:Rescaling}.
\end{proof}

The following  statement together with Theorem~\ref{T:MainResult} provide a
non-Markovian non-linear Feynman-Kac formula that
connects
maximal 
Dini subsolutions of path-dependent PDEs
with  minimal supersolutions of convex superquadratic BSDEs, for which 
we use  the notation \eqref{E:BSDEvalue}.

\begin{theorem}\label{T:BSDE}
Assume (H1) and (H2). Let $\ell$ be continuous and finite-valued. 
Let $h$ be bounded. 
Fix $(t_0,x_0)\in [0,T]\times\Omega$
and $n\in\N$. Then
  $v_n(t_0,x_0)=-\Mean_{t_0,x_0,n} [\mathcal{E}^{t_0,x_0,n}_{t_0,T}(-h)]$.
Moreover, for $P_{t_0,x_0,n}$-a.e.~$\omega\in\Omega$ and every $t\in [t_0,T]$, we have
 $v_n(t,\omega)=-\mathcal{E}^{t_0,x_0,n}_{t,T}(-h)(\omega)$.
\end{theorem}

\begin{proof}
We prove the theorem only for the case $n=1$ and $T=1$.

{\color{black}We will employ expressions $\rho^f$, $f$ being a measurable function
from $[0,1]\times\R^d$ to $\R\cup\{\infty\}$ that satisfies (H1) and (H2) in place of $\ell$, 
from Section~2 of \cite{BLT}, which are defined by}
\begin{align*}
\color{black}
\rho^f(\xi):=\sup_{Q\in\mathcal{Q}} \Mean^Q\left[\xi-\int_0^1 f(s,a^Q(s))\,ds\right],\,
\text{$\xi:\Omega\to\R$  bounded and $\cF_T$-measurable,}
\end{align*}
{\color{black}
where $\mathcal{Q}$  is the set of all probability measures on $(\Omega,\cF_T)$
that are absolutely continuous with respect to $P_{0,0}$ and
$a^Q:[0,1]\times\Omega\to\R^d$ is the unique $\F$-progressive process with
$\int_0^1 \abs{a^Q(s)}^2\,ds<\infty$, $P_{0,0}$-a.s., that satisfies}
\begin{align*}
\color{black}
dQ=\exp\left(\int_0^1 a^Q(s)\,dX(s)-\frac{1}{2}\int_0^1 \abs{a^Q(s)}^2\,ds\right)\,dP_{0,0}
\end{align*}
{\color{black}(Section~2 of \cite{BLT}).}

{\color{black} First note that} the right-hand side of \eqref{E:Rescaling} equals
$\rho^{\ell^{(t_0)}}[(-h)(x_0\odot_{t_0} X{\color{black})}]$ {\color{black}
according to (BBD) in \cite{BLT}}. 
Moreover, by
a straight-forward adjustment of the proof of
 Lemma~5.1 in  \cite{BLT}, 
$\rho^{\ell^{(t)}}[(-h)(\omega\odot_t X{\color{black})}]={\color{black} Y(t,\omega)}$
for every $t\in [t_0,T]$ and $P_{t_0,x_0}$-a.e.~$\omega\in\Omega$.
In this context, note that Lemma~5.1 and Lemma~A.2, both in \cite{BLT}, formally require continuity of
$h$, which, however, is not necessary (cf.~the corresponding material in \cite{DrapeauEtAl16AIHP}).
Using Lemma~\ref{L:Rescaling}, we can deduce that $v_1(t,\cdot)=
{\color{black}-Y(t)}$,
$P_{t_0,x_0}$-a.s., for every $t\in [t_0,T]$.
\end{proof}

\section{The second order HJB equations}\label{S:2ndOrder}

\begin{lemma} \label{L:Subsolution2}
Let $n\in\N$.
Assume (H1).
 Let $\ell=\ell(t,a)$ be continuous in $t$.
A  bounded u.s.c.~function 
$u:[0,T]\times\Omega\to\R$ 
is a   Dini 
subsolution of 
\eqref{E:TVPn} 
if and only if $u(T,\cdot)\le h$ and, for every $(t_0,x_0)\in [0,T)\times\Omega$,
 $t\in (t_0,T]$, and  $a\in\R^d$, 
\begin{align}\label{E:LemmaSubsolution2}
u(t_0,x_0)\le \Mean_{t_0,x_0,n}\left[\int_{t_0}^t \ell(s,a)\,ds+u(t,X+A^a
{\color{black}(\cdot\vee t_0)}
-A^a(t_0))\right].
\end{align}
\end{lemma}

\begin{proof}
We adapt the proof of Theorem~V.3 in \cite{Subbotina06},
which is situated in a  Markovian context with bounded control spaces,
to  our non-Markovian setting  with the unbounded control space $\R^d$.

(i)  Let $u$ be a 
Dini subsolution of \eqref{E:TVPn}. Fix $(t_0,x_0)\in [0,T)\times\Omega$. Put $\Mean=\Mean_{t_0,x_0,n}$ and $P=P_{t_0,x_0,n}$.
Assume that there are  
$t_1\in (t_0,T]$ and  $a\in\R^d$, and $\eps>0$ such that
\begin{align}\label{E:Proof:Subsolution2:I}
\Mean\left[u(t_1,X+A^a
{\color{black}(\cdot\vee t_0)}
-A^a(t_0))-u(t_0,x_0)+\int_{t_0}^{t_1} \ell(s,a)\,ds\right]<-\eps.
\end{align} 
Consider the set $S$ of all $\F$-stopping times $\tau$ that satisfy 
$t_0\le\tau\le t_1$,  
$P$-a.s., and
\begin{align}\label{E:Proof:Subsolution2:II}
\Mean\left[u(\tau,X+A^a
{\color{black}(\cdot\vee t_0)}
-A^a(t_0))-u(t_0,x_0)+\int_{t_0}^\tau \ell(s,a)\,ds\right]\ge 
\frac{\Mean[\tau]-t_0}{t_1-t_0}\cdot(-\eps).
\end{align}
We equip $S$ with an order $\preccurlyeq$ defined by $\tau_1\preccurlyeq\tau_2$ if and only if
$\tau_1\le\tau_2$, 
$P$-a.s.
Now, let $\tilde{S}$ be a totally ordered non-empty subset of $S$.
Note that there exists a sequence $(\tau_k)$ in $\tilde{S}$ such that
$\Mean[\tau_k]\uparrow \sup_{\tau\in\tilde{S}} \Mean[\tau]$ as $k\to\infty$.
Since $\tilde{S}$ is totally ordered and $\Mean$ is linear, $(\tau_k)$ is increasing
and converges to $\tilde{\tau}:=\sup_k \tau_k$. We have $\tilde{\tau}\in S$
because
\begin{align*}
&\Mean\left[u(\tilde{\tau},X+A^a
{\color{black}(\cdot\vee t_0)}
-A^a(t_0))+\int_{t_0}^{\tilde{\tau}} \ell(s,a)\,ds\right]\\&\ge
\varlimsup_k\Mean\left[u(\tau_k,X+A^a
{\color{black}(\cdot\vee t_0)}
-A^a(t_0))+\int_{t_0}^{\tau_k}\ell(s,a)\,ds\right]\\
&\ge u(t_0,x_0)+
\frac{\Mean[\tilde{\tau}]-t_0}{t_1-t_0}\cdot(-\eps)
\end{align*}
thanks to \eqref{E:Proof:Subsolution2:II} as well as to $u$ being u.s.c and bounded.
Hence, by Zorn's lemma, $S$ has a maximal element $\tau_0$.
We show that $\tau_0=t_1$, 
$P$-a.s.
To this end, assume that 
$P(\tau_0<t_1)>0$.
Define a set-valued function $M$ on 
$\color{black} \mathrm{hypo}\,u\vert_{[t_0,t_1]\times\Omega}=
\{(t,\omega,y)\in [t_0,t_1]\times\Omega\times\R:\,y\le u(t,\omega)\}$
 by
\begin{align*}
& M(t,\omega{\color{black}, y}):=\Biggl\{
\delta\in [0,
{\color{black} t_1}
-t_0]:\, t+\delta\le 
{\color{black} t_1}\text{ and}
\\
&\quad
\Mean_{t,\omega,n}\left[
u(t+\delta,X+A^a
{\color{black}(\cdot\vee t)}
-A^a(t))-{\color{black}y}+
\int_t^{t+\delta} \ell(s,a)\,ds
\right]\ge \frac{\delta}{t_1-t_0}\cdot(-\eps)
\Biggr\}.
\end{align*}
The sets $M(t,\omega{\color{black}, y})$ are non-empty because they contain $0$
and they are compact
 because, for any sequence $(\delta_k)$ in 
 $[0,
 {\color{black} t_1}
 -t]$
that converges to some $\delta$, we have
\begin{align*}
&\Mean_{t,\omega,n}\left[
u(t+\delta,X+A^a
{\color{black}(\cdot\vee t)}
-A^a(t))+
\int_t^{t+\delta} \ell(s,a)\,ds
\right]\\ &\qquad \ge
\varlimsup_k  \Mean_{t,\omega,n}
\left[
u(t+\delta_k,X+A^a
{\color{black}(\cdot\vee t)}
-A^a(t))+
\int_t^{t+\delta_k} \ell(s,a)\,ds
\right]
\end{align*}
as $u$ is u.s.c.~and bounded. 
Moreover, for every sequence $(s_k,\omega_k{\color{black}, y_k})_k$ that belongs to
${\color{black} \mathrm{hypo}\,u\vert_{[t_0,t_1]\times\Omega}}$
and that converges to some  $(t,\omega{\color{black}, y})$
with respect to the metric 
$((\tilde{s}_1,\tilde{\omega}_1{\color{black},\tilde{y}_1}),(\tilde{s}_2,\tilde{\omega}_2{\color{black},\tilde{y}_2}))
\mapsto\abs{\tilde{s}_1-\tilde{s}_2}+
\norm{\tilde{\omega}_1-\tilde{\omega}_2}_\infty{\color{black}+\abs{\tilde{y}_1-\tilde{y}_2}}$ and
hence also with respect to 
${\color{black}
((\tilde{s}_1,\tilde{\omega}_1{\color{black},\tilde{y}_1}),(\tilde{s}_2,\tilde{\omega}_2{\color{black},\tilde{y}_2}))
\mapsto\mathbf{d}_\infty
((\tilde{s}_1,\tilde{\omega}_1),(\tilde{s}_2,\tilde{\omega}_2))+\abs{\tilde{y}_1-\tilde{y}_2}
}$,
 every sequence $(\delta_k)$ in 
 $[0,
 {\color{black} t_1}
 -t_0]$ 
with $\delta_k\in M(s_k,\omega_k{\color{black}, y_k})$ has a subsequential limit that belongs to
 $M(t,\omega,{\color{black} y})$
because, for every subsequential limit  $\delta$ of $(\delta_k)$ (there is at least one), we have,
by possibly passing to a subsequence,
\begin{align*}
&\Mean_{t,\omega,n}\left[
u(t+\delta,X+A^a
{\color{black}(\cdot\vee t)}
-A^a(t))+
\int_t^{t+\delta} \ell(s,a)\,ds
\right]\\ & =
{\color{black}\Mean_{0,0}}
\Biggl[
u\left((t+\delta,\omega(\cdot\wedge t)+
\frac{1}{\sqrt{n}} 
{\color{black}(X-X(t))\bfone_{[t,T]}}{\color{black})}
+A^a
{\color{black}(\cdot\vee t)}
-A^a(t)\right)\\ &\qquad\qquad+
\int_t^{t+\delta} \ell(s,a)\,ds
\Biggr] \\&
 \ge
\varlimsup_k
{\color{black}\Mean_{0,0}}\Biggl[
u\Biggl((s_k+\delta_k,\omega_k(\cdot\wedge s_k)+
\frac{1}{\sqrt{n}} 
{\color{black}(X-X(s_k))\bfone_{[s_k,T]}}{\color{black} )}
\\ &\qquad\qquad\qquad\qquad +A^a
{\color{black}(\cdot\vee s_k)}
-A^a(s_k)\Biggr)+
\int_{s_k}^{s_k+\delta_k} \ell(s,a)\,ds
\Biggr] 
\\ &=
\varlimsup_k  \Mean_{s_k,\omega_k,n}
\left[
u(s_k+\delta_k,X+A^a
{\color{black}(\cdot\vee s_k)}
-A^a(s_k))+
\int_{s_k}^{s_k+\delta_k} \ell(s,a)\,ds
\right].
\end{align*}
Hence, $M$ considered as a map from 
${\color{black}\mathrm{hypo}\,u\vert_{[t_0,t_1]\times\Omega}}$
equipped here with the metric 
$((s_1,\omega_1{\color{black}, y_1}),(s_2,\omega_2{\color{black}, y_2}))
\mapsto\abs{s_1-s_2}+\norm{\omega_1-\omega_2}_\infty{\color{black}+\abs{y_1-y_2}}$
into the set of all compact non-empty subsets of 
$[0,
{\color{black} t_1}
-t_0]$
is u.s.c.~(Theorem~2.2 on p.~31 in \cite{Kisielewicz91DI})
and thus measurable (Theorem~2.1 on p.~29 in \cite{Kisielewicz91DI}{\color{black}), i.e.,
all sets of the form $\{(t,\omega,y)\in 
\mathrm{hypo}\,u\vert_{[t_0,t_1]\times\Omega}:
\,M(t,\omega,y)\cap E\neq\emptyset\}$, $E$ being a closed subset
of $[0,
t_1
-t_0]$,
belong to $\mathcal{B}(\mathrm{hypo}\,u\vert_{[t_0,t_1]\times\Omega})$
(see p.~41 in \cite{Kisielewicz91DI}).}
Consequently, 
by Theorem~3.13 on p.~49 in \cite{Kisielewicz91DI},
there exists a ${\color{black}\mathcal{B}(\mathrm{hypo}\,u\vert_{[t_0,t_1]\times\Omega})}$-measurable selector
$\delta(\cdot,\cdot{\color{black},\cdot})$ of $M$
that satisfies $\abs{
{\color{black} t_1}
-t_0-\delta(t,\omega{\color{black},y})}=\mathrm{dist}(
{\color{black} t_1}
-t_0,M(t,\omega{\color{black}, y}))$
for every $(t,\omega{\color{black}, y})\in {\color{black} \mathrm{hypo}\,u\vert_{[t_0,t_1]\times\Omega}}$.
Therefore, 
the map 
\begin{align*}
\omega\mapsto\delta[\omega]&:=
\delta(\tau_0(\omega),\omega(\cdot\wedge\tau_0(\omega))+A^a({\color{black}
(\cdot\vee t_0)\wedge
}\tau_0(\omega))-A^a(t_0){\color{black},}\\
&\qquad\qquad{\color{black} 
u(\tau_0(\omega),\omega(\cdot\wedge\tau_0(\omega))+A^a((\cdot\vee t_0)\wedge \tau_0(\omega))-A^a(t_0))})
\end{align*}
is $\cF_{\tau_0}$-measurable. 
Also note that, since, by \eqref{E:Subsolution2} and the continuity of $\ell(\cdot,a)$,
every set $M(t,\omega{\color{black}, u(t,\omega)})$
contains a strictly positive element, we have the inequality
$\abs{
{\color{black} t_1}
-t_0-\delta(t,\omega{\color{black}, u(t,\omega)})}<\abs{
{\color{black} t_1}
-t_0}$ 
and hence $\delta(t,\omega{\color{black}, u(t,\omega)})>0$. 
 Thus,
by Lemma~V.3 in \cite{Subbotina06},
 $\tilde{\tau}_0:=
\tau_0+\delta[\cdot]$
is an $\F$-stopping time with $\tilde{\tau}_0>\tau_0$ on $\{\tau_0<t_1\}$.
Finally, {\color{black} since}
\begin{align*}
&\Mean\left[u(\tilde{\tau}_0,X+A^a
{\color{black}(\cdot\vee t_0)}
-A^a(t_0))+\int_{t_0}^{\tilde{\tau}_0} \ell(s,a)\,ds\right]
\\&=
\int_\Omega \Biggl(\Mean\left[u(\tilde{\tau}_0,X+A^a
{\color{black}(\cdot\vee t_0)}
-A^a(t_0))+\int_{\tau_0}^{\tilde{\tau}_0} \ell(s,a)\,ds
\Bigg\vert\cF_{\tau_0}\right]
(\omega)\\ &\qquad\qquad
+\int_{t_0}^{\tau_0(\omega)} \ell(s,a)\,ds\Biggr)\,P(d\omega)
\\&=\int_\Omega 
\Biggl(\Mean_{\tau_0(\omega),\omega,n}
\left[u(\tilde{\tau}_0(\om),X+A^a
{\color{black}(\cdot\vee t_0)}
-A^a(t_0))+\int_{\tau_0(\omega)}^{\tilde{\tau}_0(\omega)} \ell(s,a)\,ds\right]
\\&\qquad\qquad+\int_{t_0}^{\tau_0(\omega)} \ell(s,a)\,ds\Biggr)\,P(d\omega)
\\&=\int_\Omega 
\Biggl(\Mean_{\tau_0(\omega),\omega+
A^a{\color{black}((\cdot\vee t_0)\wedge
\tau_0(\omega))}
-A^a(t_0),n}\\
&\qquad\qquad
\Biggl[u(\tilde{\tau}_0(\om),
X+A^a
{\color{black}(\cdot\vee \tau_0(\omega))}
-A^a(\tau_0(\omega)))
+\int_{\tau_0(\omega)}^{\tilde{\tau}_0(\omega)} \ell(s,a)\,ds\Biggr]
\\&\qquad\qquad\qquad\qquad
+\int_{t_0}^{\tau_0(\omega)} \ell(s,a)\,ds\Biggr)\,P(d\omega)
\end{align*}
{\color{black} and since, for each $\omega\in\Omega$, thanks to $\delta[\omega]\in M(\tau_0(\omega),
\omega(\cdot\wedge\tau_0(\omega))+
A^a((\cdot\vee t_0)\wedge \tau_0(\omega))
-A^a(t_0))
{\color{black}, u(\tau_0(\omega),\omega(\cdot\wedge\tau_0(\omega))+A^a((\cdot\vee t_0)\wedge \tau_0(\omega))-A^a(t_0))}
)$ and thanks to the definition of $M$,}
\begin{align*}
& \color{black} 
\Mean_{\tau_0(\omega),\omega+
A^a((\cdot\vee t_0)\wedge
\tau_0(\omega))
-A^a(t_0),n}
\Biggl[u(\tilde{\tau}_0(\om),
X+A^a
{\color{black}(\cdot\vee \tau_0(\omega))}
-A^a(\tau_0(\omega)))\\&
\color{black} 
\qquad\qquad\qquad\qquad+\int_{\tau_0(\omega)}^{\tilde{\tau}_0(\omega)} \ell(s,a)\,ds\Biggr]\\
&\color{black}=\Mean_{\tau_0(\omega),\omega(\cdot\wedge\tau_0(\omega))+
A^a((\cdot\vee t_0)\wedge \tau_0(\omega))
-A^a(t_0),n}\\
&\qquad
\Biggl[u(\tau_0(\om)+\delta[\omega],
X+A^a
{\color{black}(\cdot\vee \tau_0(\omega))}
-A^a(\tau_0(\omega)))\\&
\color{black} 
\qquad\qquad\qquad\qquad+\int_{\tau_0(\omega)}^{{\tau}_0(\omega)+\delta[\omega]} \ell(s,a)\,ds\Biggr]\\
&\color{black} \ge 
u(\tau_0(\omega),\omega(\cdot\wedge\tau_0(\omega))+A^a((\cdot\vee t_0)\wedge \tau_0(\omega))-A^a(t_0))
+\frac{\delta[\omega]}{t_1-t_0}\cdot (-\eps)\\
&\color{black} =
u(\tau_0(\omega),\omega+A^a(\cdot\vee t_0)-A^a(t_0))
+\frac{\delta[\omega]}{t_1-t_0}\cdot (-\eps),
\end{align*}
{\color{black} we have}
\begin{align*}
&{\color{black}\Mean\left[u(\tilde{\tau}_0,X+A^a
{\color{black}(\cdot\vee t_0)}
-A^a(t_0))+\int_{t_0}^{\tilde{\tau}_0} \ell(s,a)\,ds\right]}\\
&\ge \int_\Omega\Biggl(
u(\tau_0(\omega),\omega
+A^a({\color{black}\cdot\vee t_0})
-A^a(t_0))
-\frac{\delta[\omega]}{t_1-t_0} \eps 
+\int_{t_0}^{\tau_0(\omega)} \ell(s,a){\color{black}\,ds}
\Biggr)\,P(d\omega)
\\&\ge
u(t_0,x_0)+\frac{\Mean\left[\delta[\cdot]+\tau_0\right]-t_0}{t_1-t_0}\cdot(-\eps),
\end{align*}
{\color{black} where the last inequality follows from \eqref{E:Proof:Subsolution2:II} with $\tau=\tau_0$.
We can conclude that}
 $\tilde{\tau}_0\in S$.
Since  $\tilde{\tau}_0\ge \tau_0$ and $P(\{\tilde{\tau}_0>\tau_0\}\cap\{\tau_0<t_1  \})>0$, we have
$\tau_0\preccurlyeq\tilde{\tau}_0$ but not $\tau_0=\tilde{\tau}_0$, $P$-a.s.,
 i.e., we have a contradiction to $\tau_0$ being the maximal element in $S$.
Hence, $\tau_0=t_1$, $P$-a.s. However, this in turn contradicts \eqref{E:Proof:Subsolution2:I},
which concludes the proof of this direction.

(ii) Showing the remaining direction is straight-forward.
\end{proof}
\begin{theorem}\label{T:TVPn}
Assume (H1) and (H2).
Fix $n\in\N$. Let $h$ be u.s.c.~and bounded. 
Let $\ell$ be continuous and finite-valued.
Then $v_n$  
is the  unique  bounded maximal Dini subsolution of \eqref{E:TVPn}. 
Moreover, each bounded   Dini subsolution of \eqref{E:TVPn} is dominated from above by $v_n$ even if
we dispense with the assumption that $h$ is u.s.c.
\end{theorem}
\begin{proof}
(i) Existence and regularity: 
First, we show that $v_n$ is bounded and u.s.c.
Boundedness of 
$h$, (H1),  and  (H2)  
yield boundedness 
of $v_n$. 
{\color{black}
To establish upper semi-continuity, } 
we will use $\odot$ defined by  \eqref{E:odot} and the notation \eqref{E:ellt0}.
For this reason, we  assume that $T=1$ and $n=1$ 
 but this assumption is not restrictive.
{\color{black} We continue by fixing  a pair} $(t_0,x_0)\in [0,T]\times\Omega$ {\color{black}and considering} 
a  sequence $(t_k,x_k)_{k\ge 1}$ in $[0,T]\times\Omega$
that converges to $(t_0,x_0)$ and  satisfies
$\lim_k v_n(t_k,x_k)=\varlimsup_{(t,x)\to (t_0,x_0)} v_n(t,x)$.
{\color{black} We distinguish between two cases.

\textit{Case~1:} $t_0<1$. Fix  an $\eps<0$ such that $t_0+\eps<1$. We can assume, without loss of generality,
that $t_k<t_0+\eps$ for all $k\in\N$.
Let $a\in\mathcal{L}_b$. 
Then the map (see \eqref{E:ellt0})
\begin{align*}
(s,t,\omega)\mapsto \ell^{(s)}(t, a(t,\omega))=
(1-s)\,\ell\left(s+t(1-s),\frac{a(t,\omega)}{\sqrt{1-s}}\right)
\end{align*}
from $[0,t_0+\eps]\times [0,T]\times\Omega$ to $\R$
has a real upper bound
 due to $\ell$ being continuous
and $a$ being bounded.}
Note that
 $(t,x)\mapsto (x\odot_t \omega)$, $[0,T)\times\Omega\to\Omega$, is continuous
for every $\omega\in\Omega$ with $\omega(0)=0$.
Thus{\color{black}, by} 
Lemma~\ref{L:Rescaling}, 
\begin{align*}
&\lim_k v_n(t_k,x_k)\le 
{\color{black} \varlimsup_k \Mean_{0,0,n} \left[
h(x_k\odot_{t_k} [X+A^a])
+\int_0^T \ell^{(t_k)} (t,a(t))\,dt
\right]}\\
&\le 
\Mean_{0,0,n} \left[
h(x_0\odot_{t_0} [X+A^a])
+\int_0^T {\color{black}\ell^{(t_0)} (t,a(t))}
\,dt
\right].
\end{align*}
{\color{black} Since $a\in\mathcal{L}_b$ was arbitrary, we can deduce after again invoking 
Lemma~\ref{L:Rescaling} that $\lim_k v_n(t_k,x_k)\le v_n(t_0,x_0)$.} 

{\color{black}\textit{Case~2:} $t_0=1$. Then, proceeding similarly as in Case~1 but using the constant
control $(t,\omega)\mapsto a(t,\omega)=0$, we have}
\begin{align*}
& \color{black} \lim_k v_n(t_k,x_k)\le 
{\color{black} \varlimsup_k \Mean_{0,0,n} \left[
h(x_k\odot_{t_k} [X+A^a])
+\int_0^T \ell^{(t_k)} (t,0)\,dt
\right]}\\
& \color{black} \le 
h(x_0)=v_n(t_0,x_0).
\end{align*}
{\color{black} We can conclude that } $v_n$ is u.s.c.

Next, we establish the  subsolution property via the BSDE connection in Theorem~\ref{T:BSDE}.
To this end, we use the notation~\eqref{E:BSDEvalue} and fix $(t_0,x_0)\in [0,T]\times\Omega$,
$t\in [t_0,T]$, and $a\in\R^d$.
Note first that,
by Proposition~3.6~(1) in \cite{DrapeauEtAl13AOP}, 
\begin{align}\label{E:BSDE:DPP}
\mathcal{E}^{t_0,x_0,n}_{t_0,T}(-h)=\mathcal{E}^{t_0,x_0,n}_{t_0,t}\left(
\mathcal{E}^{t_0,x_0,n}_{t,T}(-h)
\right).
\end{align}
Since $v_n$ is bounded and, by Theorem~\ref{T:BSDE}, $v_n(t,\cdot)=-\mathcal{E}^{t_0,x_0,n}_{t,T}(-h)$,
$P_{t_0,x_0,n}$-a.s., we can apply
Theorem~3.4 in \cite{DrapeauEtAl16AIHP} (see also Lemma~A.2 in \cite{BLT}) together with (3) in \cite{BLT}
to  deduce that, with  $Q^a_{t_0,x_0,n}:=\Prob_{t_0,x_0,n}\circ (X+A^a
{\color{black}(\cdot\vee t_0)}
-A^a(t_0))^{-1}$,
\begin{align*}
&\Mean_{t_0,x_0,n}\left[
\mathcal{E}^{t_0,x_0,n}_{t_0,t}\left(
\mathcal{E}^{t_0,x_0,n}_{t,T}(-h)\right)
\right]\ge 
\Mean^{Q^a_{t_0,x_0,n}} \left[\mathcal{E}^{t_0,x_0,n}_{t,T}(-h)-\int_{t_0}^t \ell(s,a)\,ds\right]\\
&=\Mean_{t_0,x_0,n} \left[\mathcal{E}^{t_0,x_0,n}_{t,T}(-h)(X+A^a{\color{black}(\cdot\vee t_0)}
-A^a(t_0))-\int_{t_0}^t \ell(s,a)\,ds\right].
\end{align*}
Thus, by  \eqref{E:BSDE:DPP} and
Theorem~\ref{T:BSDE}, $v_n$ satisfies \eqref{E:LemmaSubsolution2} in place of $u$.
Hence, by Lemma~\ref{L:Subsolution2}, $v_n$ is a Dini subsolution of \eqref{E:TVPn}.

(ii) Uniqueness: Let $u$ be a bounded Dini subsolution of \eqref{E:TVPn}. Without loss of generality, let $n=1$.
Fix $(t_0,x_0)\in [0,T)\times\Omega$ and $\eps>0$. It suffices to show that $u(t_0,x_0)\le 
{\color{black}v_n}(t_0,x_0)+\eps$.
By (a slight modification) of (36) and (BBD), both in \cite{BLT}, there exists an $a\in\mathcal{L}_b^{t_0}$ 
with $a(t,\omega)=\sum_{i=1}^m a_i(\omega).\bfone_{[s_{i-1},s_i)}(t)$, where $t_0=s_0<s_1<\cdots<s_m=T$
and each $a_i:\Omega\to\R^d$ is $\cF_{s_{i-1}}$-measurable,
such that $\tilde{\Mean}_{t_0,x_0} \left[\int_{t_0}^T \ell(t,a(t))\,dt+h(X)\right]\le v_n(t_0,x_0)+\eps$.
Here, for any $(s,x)\in [t_0,T]\times\Omega$, 
 $\tilde{P}_{s,x}:=P_{s,x}\circ(\tilde{X}^{s,x})^{-1}$,
$\tilde{\Mean}_{{\color{black} s,x}}:=\Mean^{\tilde{P}^{s,x}}$, and
$\tilde{X}^{s,x}:[0,T]\times\Omega\to\R^d$ is the unique solution of
$\tilde{X}^{s,x}(t)=\tilde{X}^{s,x}(s)+\int_s^t a(r,\tilde{X}^{s,x})\,dr +X(t)-X(s)$, $t\in [s,T]$,
with initial condition $\tilde{X}^{s,x}\vert_{[0,s]}=x\vert_{[0,s]}$.
Next, note that $a_1\circ\tilde{X}^{t_0,x_0}=a_1\circ x_0=:\tilde{a}_1\in\R^d$ and
$X+A^{\tilde{a}_1}{\color{black}(\cdot\vee t_0)}
-A^{\tilde{a}_1}(t_0) =\tilde{X}^{t_0,x_0}$ on $[0,s_1]$, $P_{t_0,x_0}$-a.s.
Thus, by Lemma~\ref{L:Subsolution2},
\begin{align*}
u(t_0,x_0)&\le\Mean_{t_0,x_0}\left[\int_{t_0}^{s_1} \ell(t,\tilde{a}_1)\,dt +u(s_1,\tilde{X}^{t_0,x_0})\right]\\
&=\tilde{\Mean}_{t_0,x_0}\left[\int_{t_0}^{s_1} \ell(t,a(t))\,dt+u(s_1,X)\right].
\end{align*}
Also note that, for $P_{t_0,x_0}$-a.e.~$\omega\in\Omega$, we have, with $x_1=\tilde{X}^{t_0,x_0}(\omega)$,
$a_2\circ\tilde{X}^{s_1,x_1}=a_2\circ x_1\in\R^d$
and $X+A^{a_2\circ x_1}{\color{black}(\cdot\vee s_1)}
-A^{a_2\circ x_1}(s_1)=\tilde{X}^{s_1,x_1}$ on $[0,s_2]$, $P_{s_1,x_1}$-a.s.,
and thus, by Lemma~\ref{L:Subsolution2},
\begin{align*}
u(s_1,x_1)&\le\Mean_{s_1,x_1}\left[\int_{s_1}^{s_2} \ell(t,\tilde{a}_2)\,dt +u(s_2,\tilde{X}^{s_1,x_1})\right]\\
&=\tilde{\Mean}_{s_1,x_1}\left[\int_{s_1}^{s_2} \ell(t,a(t))\,dt+u(s_2,X)\right].
\end{align*}
Therefore
\begin{align*}
\tilde{\Mean}_{t_0,x_0}\left[u(s_1,X)\right]&\le
\int_\Omega
\tilde{\Mean}_{s_1,x_1}\left[\int_{s_1}^{s_2} \ell(t,a(t))\,dt+u(s_2,X)\right]
\Bigg\vert_{x_1=\omega}
\,\tilde{P}_{t_0,x_0}(d\omega)\\
&=\tilde{\Mean}_{t_0,x_0}\left(
\tilde{\Mean}_{t_0,x_0}\left[
\int_{s_1}^{s_2} \ell(t,a(t))\,dt+u(s_2,X)
\Bigg\vert\cF_{s_1}
\right]
\right)\\
&=\tilde{\Mean}_{t_0,x_0}\left[\int_{s_1}^{s_2} \ell(t,a(t))\,dt+u(s_2,X)\right].
\end{align*}
Repeating this procedure and noting that 
$u(T,\cdot)\le h$ yields 
\begin{align*}
u(t_0,x_0)\le \tilde{\Mean}_{t_0,x_0}  \left[\int_{t_0}^T \ell(t,a(t))\,dt+h(X)\right]\le v_n(t_0,x_0)+\eps,
\end{align*}
which concludes the proof.
\end{proof}

\section{The first order HJB equation}\label{S:1stOrder}

\begin{lemma} \label{L:MinimaxSuper}
Assume (H1). 
An l.s.c.~function $u:[0,T]\times\Omega\to\R\cup\{\infty\}$ 
bounded from below is a  
 minimax supersolution of \eqref{E:TVP0} if and only if
$u(T,\cdot)\ge h$ and, for every $(t_0,x_0)\in [0,T)\times\Omega$ 
and 
$t\in (t_0,T]$, there exists an $x\in\mathcal{X}^{1,1}(t_0,x_0)$ such that
\begin{align}\label{E:LemmaSupersolution}
u(t_0,x_0)\ge \int_{t_0}^t \ell(s,x^\prime(s))\,ds +u(t,x).
\end{align}
\end{lemma}

\begin{proof} 
(i)
Let $u$ be a minimax supersolution of \eqref{E:TVP0}.
Assume that there exist a $(t_0,x_0)\in\mathrm{dom}(u)$ with $t_0<T$ and a $t_1\in (t_0,T]$ such that, 
for all $x\in\mathcal{X}^{1,1}(t_0,x_0)$,
$u(t_1,x)-u(t_0,x_0)+\int_{t_0}^{t_1} \ell(s,x^\prime(s))\,ds>0$.
Then
there is an $\eps>0$ such that 
\begin{align}\label{E:ProofLemmaSupersolution}
u(t_1,x)-u(t_0,x_0)+\int_{t_0}^{t_1} \ell(s,x^\prime(s))\,ds>\eps
\end{align}
for all $x\in\mathcal{X}^{1,1}(t_0,x_0)$
because  
$\inf_{x\in\mathcal{X}^{1,1}(t_0,x_0)} \left[\int_{t_0}^{t_1} \ell(s,x^\prime(s))\,ds+u(t_1,x)\right]$
has a minimizer {\color{black} in $\R\cup\{\infty\}$}  ({\color{black}this can be shown exactly as the corresponding
statement for the non-path-dependent counterpart Theorem~11.1.i in 
Section~11.1 of \cite{Cesari} and its extension to unbounded domains, which is
treated in Section~11.2 of} \cite{Cesari}).
Next, consider the (non-void) set $S$ of all $(t,x)\in [t_0,t_1)\times\mathcal{X}^{1,1}(t_0,x_0)$ for which 
$u(t,x)-u(t_0,x_0)+\int_{t_0}^t \ell(s,x^\prime(s))\,ds\le\frac{t-t_0}{t_1-t_0}\cdot\eps$
holds. Let $\tilde{s}:=\sup\{t:(t,x)\in S\text{ for some $x$}\}$.
 This supremum is attained.
To see this, consider  a sequence $(s_n,x_n)_n$ 
in $S$ with $s_n\to \tilde{s}$.
Since $\sup_n\int_{t_0}^{s_n} \ell(s,x_n^\prime(s))\,ds<\infty$ and $u$ is bounded from below,
one can show as in Sections 11.1 and 11.2 of \cite{Cesari} that   there is a subsequence $(s_{n_k},x_{n_k})_k$ such that $(x_{n_k})$ converges to some
$\tilde{x}$ in $\mathcal{X}^{1,1}(t_0,x_0)$, 
i.e.,  $\norm{x_{n_k}-\tilde{x}}_\infty+\norm{x_{n_k}^\prime-\tilde{x}^\prime}_{L^1(t_0,T;\R^d)}\to 0$.
Lower semi-continuity of $u$ together with Theorem~10.8.ii in \cite{Cesari} yield
\begin{align*}
u(\tilde{s},\tilde{x})+\int_{t_0}^{\tilde{s}} \ell(s,\tilde{x}^\prime(s))\,ds  &\le
\varliminf_k u(s_{n_k},x_{n_k})+\varliminf_k
 \int_{t_0}^{s_{n_k}}  \ell(s,x_{n_k}^\prime(s))\,ds
\\&
\le  \frac{\tilde{s}-t_0}{t_1-t_0}\cdot\eps +u(t_0,x_0).
\end{align*}
Thus  $(\tilde{s},\tilde{x})\in\mathrm{dom}(u)$ and,
 by \eqref{E:ProofLemmaSupersolution}, $\tilde{s}<t_1$. 
 Hence, by \eqref{E:Supersolution}, there is 
 a $\delta>0$ and an $x\in\mathcal{X}^{1,1}(\tilde{s}, \tilde{x})$ such that
$u(\tilde{s}+\delta,x)-u(\tilde{s},\tilde{x})
+\int_{\tilde{s}}^{\tilde{s}+\delta}\ell(s,x^\prime(s))\,ds \le \frac{\delta}{t_1-t_0}\cdot\eps$
and  $\tilde{s}+\delta<t_1$.
Thus $(\tilde{s}+\delta,x)\in S$, which is a contradiction to the maximality of $\tilde{s}$.

(ii) 
Fix $(t_0,x_0)\in\mathrm{dom}(u)$ with $t_0<T$.
Suppose that, for every $(\tilde{t}_0,\tilde{x}_0)\in [t_0,T)\times\mathcal{X}^{1,1}(t_0,x_0)$ 
for every $t\in (\tilde{t}_0,T]$, there exists an $x\in\mathcal{X}^{1,1}(\tilde{t}_0,\tilde{x}_0)$ such that \eqref{E:LemmaSupersolution} holds with $(t_0,x_0)$ replaced by $(\tilde{t}_0,\tilde{x}_0)$.
 Then one can proceed similarly 
as in the proof of Lemma~3.6
in \cite{BK18JFA} to show that there exists a
sequence $(x_n)$ in
 $\mathcal{X}^{1,1}(t_0,x_0)$  and an increasing
  sequence $(A_n)$ of finite  subsets of $[t_0,T]$ whose union $A$ is dense in 
 $[t_0,T]$
 such that, for each $n\in\N$ and every $t\in A_n$,  \eqref{E:LemmaSupersolution} holds with $x$ replaced by $x_n$.
Thus $\sup_{t\in A,n\in\N} \int_{t_0}^t \ell(s,x_n^\prime(s))\,ds\le u(t_0,x_0)+c<\infty$, where $-c$ is a lower bound of $u$.
Note that we are in a similar situation as in part (i) of this proof and thus it can be shown in the same way 
that there is a subsequence $(x_{n_k})$ of $(x_n)$
that converges to some $\tilde{x}$ in $\mathcal{X}^{1,1}(t_0,x_0)$.
Now, fix $t\in (t_0,T]$ and  a sequence $(s_k)$ in $A$ with $s_k\in A_{n_k}$ for each $k$
and with $s_k\to t$ as $k\to\infty$. By  Theorem~10.8.ii in \cite{Cesari} and lower semi-continuity of $u$,
\begin{align*}
u(t_0,x_0)&\ge \varliminf_k\left[
\int_{t_0}^{s_k} \ell(s,x_{n_k}^\prime(s))\,ds + u(s_k,x_{n_k})
\right]\ge \int_{t_0}^t \ell(s,\tilde{x}^\prime(s)\,ds+u(t,\tilde{x}),
\end{align*}
i.e., there is an $x\in\mathcal{X}^{1,1}(t_0,x_0)$ such that,
for every $t\in [t_0,T]$, \eqref{E:LemmaSupersolution} holds. From this point,
 \eqref{E:Supersolution} follows easily.
\end{proof}

\begin{lemma} \label{L:MinimaxSub}
Assume (H1). 
An l.s.c.~function $u:[0,T]\times\Omega\to\R\cup\{\infty\}$ is an
 l.s.c.~minimax subsolution of \eqref{E:TVP0} if and only if
$u(T,\cdot)\le h$ and, for every $(t_0,x_0)\in [0,T)\times\Omega$, $t\in (t_0,T]$, and $x\in\mathcal{X}^{1,1}(t_0,x_0)$,
we have
\begin{align}\label{E:LemmaSubsolution}
u(t_0,x_0)\le \int_{t_0}^t \ell(s,x^\prime(s))\,ds +u(t,x).
\end{align}
\end{lemma}

\begin{proof}
(i) 
Let $u$ be an l.s.c.~minimax subsolution of \eqref{E:TVP0}. For the sake of a contradiction,
assume that 
 there exist a $(t_0,x_0)\in [0,T)\times\Omega$, a $t_1\in (t_0,T]$, an $x\in\mathcal{X}^{1,1}(t_0,x_0)$,
and an $\eps>0$ such that
\begin{align}\label{E:ProofLemmaSubsolution}
u(t_0,x_0)-u(t_1,x)-\int_{t_0}^{t_1}\ell(s,x^\prime(s))\,ds>\eps
\end{align}
as well as $(t_1,x)\in\mathrm{dom}(u)$ 
and  $\int_t^{t_1} \ell(s,x^\prime(s))\,ds <\infty$  for all $t\in [t_0,t_1]$.
 Put
\begin{align*}
\tilde{s}:=\inf\left\{t\in (t_0,t_1]: u(t,x)-u(t_1,x)-\int_t^{t_1} \ell(s,x^\prime(s))\,ds\le \frac{t_1-t}{t_1-t_0}\cdot\eps\right\}.
\end{align*}
We show that this infimum is attained. To this end, consider a sequence $(s_n)$ in $(t_0,t_1]$ with
$s_n\to \tilde{s}$ and 
$u(s_n,x)-u(t_1,x)-\int_{s_n}^{t_1}\ell(s,x^\prime(s))\,ds \le \frac{t_1-s_n}{t_1-t_0}\cdot\eps$. Then
\begin{align*}
&u(\tilde{s},x)-\int_{\tilde{s}}^{t_1} \ell(s,x^\prime(s))\,ds\le
\varliminf_n u(s_n,x)-\lim_n  \int_{s_n}^{t_1} \ell(s,x^\prime(s))\,ds\\
&\le \varliminf_n \left[u(s_n,x)- \int_{s_n}^{t_1} \ell(s,x^\prime(s))\,ds\right]
\le \frac{t_1-\tilde{s}}{t_1-t_0}\cdot\eps+u(t_1,x),
\end{align*}
i.e., $\tilde{s}$ is a minimum and $(\tilde{s},x)\in\mathrm{dom}(u)$. 
Moreover, by \eqref{E:ProofLemmaSubsolution}, $t_0<\tilde{s}\le t_1$.
Finally, by \eqref{E:Subsolution}, there is  a $\delta\in (0,\tilde{s}-t_0]$ such that
$u(\tilde{s}-\delta,x)-u(\tilde{s},x)-\int_{\tilde{s}-\delta}^{\tilde{s}}\ell(s,x^\prime(s))\,ds\le
\frac{\delta}{t_1-t_0}\cdot\eps$. Hence, 
$u(\tilde{s}-\delta,x)-u(t_1,x)-\int_{\tilde{s}-\delta}^{t_1}\ell(s,x^\prime(s))\,ds
\le \frac{t_1-(\tilde{s}-\delta)}{t_1-t_0}\cdot\eps$,
which is a contradiction to the minimality of $\tilde{s}$.

(ii) Showing the remaining direction is straight-forward.
\end{proof}

\begin{proof}[Proof of Theorem~\ref{T:1stOrder}]
(a) Establishing the lower semi-continuity of $v_0$ is quite standard.
It is very similar to the proof of Lemma~\ref{L:LSC} and actually slightly easier as
no probability is involved (cf.~also Proposition 3.1 in \cite{DM-F00} for the non-path-dependent case).
To deduce that $v_0$ is an l.s.c.~minimax solution,
it suffices to apply the dynamic programming principle
with the existence of a minimizer for $(\text{DOC})$ 
(cf.~{\color{black} Theorem~11.1.i in} \cite{Cesari})
together with Lemmata~\ref{L:MinimaxSuper}
 and \ref{L:MinimaxSub}.
Finally, we can apply Lemmata~\ref{L:MinimaxSuper}
 and \ref{L:MinimaxSub} again to obtain a comparison principle
between l.s.c.~minimax subsolutions and minimax supersolutions,
from which uniqueness follows.

(b) Taking Remark~\ref{R:NoLavrentiev} into account, one can see that the proof follows essentially from the content of Section~\ref{S:2ndOrder}. 
The measures $\Prob_{t_0,x_0,n}$ need to be replaced by the Dirac measures
under which $X=x(\cdot\wedge t_0)$ a.s.~for each $(t_0,x_0)\in [0,T]\times\Omega$ and one should note that
 the domains of the controls are different ($[0,T]$ here vs. $[0,T]\times\Omega$
in Section~\ref{S:2ndOrder}). Moreover, the BSDE argument in part (i) of the proof of
Theorem~\ref{T:TVPn}
needs to be replaced by the deterministic dynamic programming principle.
\end{proof}

{\color{black} \section{Conclusion}}

{\color{black}
The main contributions of this paper are 
a non-Markovian vanishing viscosity result for path-dependent PDEs (PPDEs) that corresponds to
 the non-exponential Schilder theorem in \cite{BLT},
 well-posedness for new notions of generalized solutions of PPDEs that can have quadratic or even
super-quadratic growth in the gradient, and 
a non-Markovian Feynman-Kac formula for convex superquadratic BSDEs.

We want to emphasize that,
here,  control-theoretic methods, or  equivalently (in our case), results from the theory of large deviations
have been applied
to obtain stability results for PPDEs (corresponding PDEs results have been obtained 
in \cite{BLT} in the Markovian case).  Of great interest would be research that investigates the opposite direction, i.e.,
to establish stronger results
(in particular, suitable stability results)
for PPDEs with (only) quadratic growth in the gradient in order to  derive non-Markovian large
deviation results similarly as it has been successfully done in the Markovian case via PDEs.
}

\bibliographystyle{amsplain}
\bibliography{BK19}
\end{document}